\documentclass[11pt,twoside, leqno]{article}

\usepackage{mathrsfs}
\usepackage{amssymb}
\usepackage{amsmath}
\usepackage{amsthm}
\allowdisplaybreaks

\def\rr{{\mathbb R}}
\def\rn{{{\rr}^n}}
\def\rnz{{\rr}^{n+1}_+}

\def\rq{{\rr}^{n+1}_+}
\def\zz{{\mathbb Z}}

\def\cc{{\mathbb C}}
\def\cn{{\mathbb N}}
\def\cs{{\mathcal S}}

\def\cl{{\mathfrak{L}}}

\def\cp{{{\mathcal P}}}
\def\cps{{{\mathcal P}_s(\rn)}}
\def\cm{{\mathcal M}}
\def\car{{\mathcal R}}

\def\ca{{\mathcal A}}
\def\ccc{{\mathcal C}}

\def\fz{\infty}
\def\az{\alpha}
\def\supp{{\mathop\mathrm{\,supp\,}}}

\def\loc{{\mathop\mathrm{\,loc\,}}}

\def\lz{\lambda}
\def\dz{\delta}

\def\ez{\epsilon}

\def\bz{\beta}
\def\ro{\rho}

\def\gz{{\gamma}}
\def\oz{{\omega}}

\def\tz{\theta}
\def\sz{\sigma}

\def\wz{\widetilde}
\def\rz{\rightarrow}

\def\hs{\hspace{0.3cm}}

\def\ls{\lesssim}
\def\gs{\gtrsim}

\def\fss{\frac nm(\frac{1}{\wz p_0(\oz)}-1)}

\def\bbmo{{\mathrm{BMO}}}
\def\vmo{{\mathrm{VMO}}_{\ro,L}(\rn)}

\def\vvmo{{\mathrm{VMO}^s_{\ro,L}(\rn)}}
\def\bmor{{\mathrm{BMO}_{\ro,L}(\rn)}}
\def\vmor{{\mathrm{VMO}^{s_0}_{\ro,\,L}(\rn)}}
\def\bmoz{{\mathrm{BMO}_{\ro,L^\ast}(\rn)}}
\def\vmoz{{\mathrm{VMO}_{\ro,L^\ast}(\rn)}}

\def\bmoq{{\mathrm{BMO}^{q,s}_{\ro,L}(\rn)}}

\def\hw{{H_{\oz,L}(\rn)}}

\def\hw{{H_{\oz,L}(\rn)}}

\def\hv{{B^{s,s_1}_{\oz,L}(\rn)}}
\def\hz{{H_{\oz,\,L^\ast}(\rn)}}
\def\twl{{T_{\oz}({\rr}^{n+1}_+)}}
\def\twc{{ T^{\fz}_{\oz,c}({\rr}^{n+1}_+)}}
\def\twz{{ T^{\fz}_{\oz}({\rr}^{n+1}_+)}}
\def\tw0{{ T^{\fz}_{\oz,0}({\rr}^{n+1}_+)}}
\def\twv{{T^{\fz}_{\oz,\mathrm v}({\rr}^{n+1}_+)}}
\def\wtw{{\widetilde{T}_{\oz}({\rr}^{n+1}_+)}}

\def\com{\complement}
\def\r{\right}
\def\lf{\left}
\def\lfr{\lfloor}
\def\rf{\rfloor}
\def\la{\langle}
\def\ra{\rangle}

\newtheorem{thm}{Theorem}[section]
\newtheorem{lem}{Lemma}[section]
\newtheorem{prop}{Proposition}[section]
\newtheorem{rem}{Remark}[section]
\newtheorem{cor}{Corollary}[section]
\newtheorem{defn}{Definition}[section]

\numberwithin{equation}{section}

\begin{document}

\arraycolsep=1pt

\title{{\vspace{-5cm}\small\hfill\bf J. Fourier Anal. Appl., to appear}\\
\vspace{4cm}\bf Predual Spaces of Banach Completions of Orlicz-Hardy Spaces
Associated with Operators\footnotetext{\hspace{-0.35cm} 2000 {\it Mathematics
Subject Classification}. Primary 42B35; Secondary 42B30.
\endgraf{\it Key words and phrases.} operator, Orlicz function, Orlicz-Hardy space,
VMO, predual space, Banach completion, tent space, molecule.
\endgraf
Dachun Yang is supported by the National
Natural Science Foundation (Grant No. 10871025) of China.
\endgraf $^\ast\,$Corresponding author.}}
\author{Renjin Jiang and Dachun Yang$\,^\ast$}
\date{ }
\maketitle

\begin{center}
\begin{minipage}{13.5cm}\small
{\noindent{\bf Abstract.} Let $L$ be a linear operator in
$L^2({{\mathbb R}^n})$ and generate an analytic semigroup
$\{e^{-tL}\}_{t\ge 0}$ with kernels satisfying an upper bound of
Poisson type, whose decay is measured by $\theta(L)\in (0,\infty].$
Let $\omega$ on $(0,\infty)$ be of upper type $1$ and of critical
lower type $\widetilde p_0(\omega)\in (n/(n+\theta(L)), 1]$ and
$\rho(t)={t^{-1}}/\omega^{-1}(t^{-1})$ for $t\in (0,\infty)$. In
this paper, the authors first introduce the VMO-type space
$\mathrm{VMO}_{\rho,L}({\mathbb R}^n)$ and the tent space
$T^{\infty}_{\omega,\mathrm v}({\mathbb R}^{n+1}_+)$ and
characterize the space $\mathrm{VMO}_{\rho,L}({\mathbb R}^n)$ via
the space $T^{\infty}_{\omega,\mathrm v}({{\mathbb R}}^{n+1}_+)$.
Let $\widetilde{T}_{\omega} ({{\mathbb R}}^{n+1}_+)$ be the Banach
completion of the tent space $T_{\omega}({\mathbb R}^{n+1}_+)$. The
authors then prove that $\widetilde{T}_{\omega}({\mathbb
R}^{n+1}_+)$ is the dual space of $T^{\infty}_{\omega,\mathrm
v}({\mathbb R}^{n+1}_+)$. As an application of this, the authors
finally show that the dual space of
$\mathrm{VMO}_{\rho,L^\ast}({\mathbb R}^n)$ is the space
$B_{\omega,L}({\mathbb R}^n)$, where $L^\ast$ denotes the adjoint
operator of $L$ in $L^2({\mathbb R}^n)$ and $B_{\omega,L}({\mathbb
R}^n)$ the Banach completion of the Orlicz-Hardy space
$H_{\omega,L}({\mathbb R}^n)$. These results generalize the known
recent results by particularly taking $\omega(t)=t$ for $t\in (0,\infty)$.}
\end{minipage}
\end{center}

\vspace{0cm}

\section{Introduction\label{s1}}

\hskip\parindent The space $\mathrm{VMO}(\rn)$
(the space of functions with vanishing mean
oscillation) was first studied by Sarason \cite{sa}.
Coifman and Weiss \cite{cw} introduced the space $\mathrm{CMO}(\rn)$
which is defined to be the closure in the
BMO norm of the space of continuous functions with
compact support, and moreover, proved that the space $\mathrm{CMO}(\rn)$ is
the predual of the Hardy space $H^1(\rn)$. When $p<1$,
Janson \cite{ja} introduced the space $\lz_{n(1/p-1)}(\rn)$ which
is defined to be the closure of the space of Schwartz functions
in the norm of the Lipschitz space $\Lambda_{n(1/p-1)}(\rn)$, and proved that
$(\lz_{n(1/p-1)}(\rn))^\ast=B^p(\rn)$, where $B^p(\rn)$ is the Banach
completion of the Hardy space $H^p(\rn)$; see also \cite{pe,wws} for more
properties about the space $\lz_{n(1/p-1)}(\rn)$.

In recent years, the study of function spaces associated with
operators has inspired great interests; see, for example,
\cite{adm2,amr,ar,ddsty, ddy,dxy,dy1,dy2,hm1,ya1,ya2} and their references. Let $L$ be
a linear operator in $L^2(\rn)$ and generate an analytic semigroup
$\{e^{-tL}\}_{t\ge 0}$ with kernels satisfying an upper bound of
Poisson type, whose decay is measured by $\tz(L)\in (0,\fz].$
Auscher, Duong and McIntosh \cite{adm2} introduced the Hardy space
$H^1_L(\rn)$ by using the Lusin-area function and established its
molecular characterization. Duong and Yan \cite{dy1, dy3}, and Duong,
Xiao and Yan \cite{dxy} introduced and studied some $\mathrm{BMO}$ spaces and
Morrey-Campanato spaces associated with operators. Duong
and Yan \cite{dy2} further proved that the dual space of the Hardy space
$H^1_L(\rn)$ is the space $\bbmo_{L^\ast}(\rn)$ introduced in \cite{dy1},
where $L^\ast$ denotes the adjoint operator of $L$ in $L^2(\rn).$ Yan \cite{ya2}
generalized all these results to the Hardy spaces
$H^p_L(\rn)$ with $p\in (n/(n+\tz(L)), 1]$ and their dual spaces.
Moreover, recently, Deng, Duong et al in \cite{ddsty}
introduced the space $\mathrm{VMO}_L(\rn)$ and proved that
$(\mathrm{VMO}_{L^\ast}(\rn))^\ast=H_L^1(\rn)$.

On the other hand, the Orlicz-Hardy space was studied by Janson \cite{ja} and
Viviani \cite{v}. Let $\oz$ on $(0,\fz)$ be of upper type $1$ and of
critical lower type $\wz p_0(\oz)\in (n/(n+\tz(L)), 1]$ and
$\rho(t)\equiv{t^{-1}}/\oz^{-1}(t^{-1})$ for $t\in (0,\fz).$
The Orlicz-Hardy space $\hw$ and its dual space $\bmoz$
associated with the aforementioned operator $L$ and its dual operator $L^\ast$
in $L^2(\rn)$ were introduced in \cite{jyz}. If $\oz(t)=t^p$ for all $t\in (0,\fz)$,
then $\hw=H^p_L(\rn)$ and $\bmoz$ becomes $\mathrm{BMO}_{L^\ast}(\rn)$ when $p=1$
or the Morrey-Campanato space (see \cite{dy3}) when $p<1$.
The main purpose of this paper
is to study the predual space of the Banach completion of the
Orlicz-Hardy space $\hw$.

In fact, in this paper, we first introduce the VMO-type space
$\mathrm{VMO}_{\ro,\,L}(\rn)$ and the tent space
$T^{\fz}_{\oz,\mathrm v}({\rr}^{n+1}_+)$ and characterize the space
$\mathrm{VMO}_{\ro,\,L}(\rn)$ via the space $T^{\fz}_{\oz,\mathrm
v}({\rr}^{n+1}_+)$. Let $\widetilde{T}_{\oz}({\rr}^{n+1}_+)$ be the
Banach completion of the tent space $T_{\oz}({\rr}^{n+1}_+)$. We
then prove that $\widetilde{T}_{\oz}({\rr}^{n+1}_+)$ is the dual
space of $T^{\fz}_{\oz,\mathrm v}({\rr}^{n+1}_+)$. As an application
of this, we finally show that the dual space of
$\mathrm{VMO}_{\ro,\,L^\ast}(\rn)$ is the space $B_{\oz,\,L}(\rn)$,
where $L^\ast$ denotes the adjoint operator of $L$ in $L^2({\mathbb
R}^n)$ and $B_{\oz,\,L}(\rn)$ the Banach completion of the
Orlicz-Hardy space $H_{\oz,\,L}(\rn)$. In particular, if $p\in (0, 1]$
and $\oz(t)=t^p$ for all $t\in (0,\fz)$,
we obtain the predual space of the Banach completion of the
Hardy space $H_L^p(\rn)$ in \cite{ya2}, and if $p=1$, we
re-obtain that $(\mathrm{VMO}_{L^\ast}(\rn))^\ast=H_L^1(\rn)$,
which is the main result in \cite{ddsty}.
Moreover, we prove that if $L=\Delta$, $p\in (0,1]$ and
$\omega(t)=t^p$ for all $t\in (0,\fz)$,
the space $\mathrm{VMO}_{\rho,L}({\mathbb R}^n)$
coincides with the space $\lambda_{n(1/p-1)}({\mathbb R}^n)$ in \cite{ja}
(see also \cite{pe, wws}), where $\Delta=-\sum^n_{i=1}\frac {\partial^2}{\partial x_i^2}$
is the Laplace operator on $\rn$.

Precisely, this paper is organized as follows. In Section \ref{s2}, we
recall some known definitions and notation concerning aforementioned
operators, Orlicz functions, the Orlicz-Hardy spaces and BMO spaces associated
with these operators and describe some basic assumptions on the operator $L$
and the Orlicz function $\oz$ considered in this paper. We remark that
there exist many operators satisfying these assumptions
(see \cite{ddsty,dxy, dy2, ya2, jyz} for examples
of such operators). Also, if $p\in (0,1]$, then $\oz(t)=t^p$ for
all $t\in (0,\fz)$ is a typical example of Orlicz functions
satisfying our assumptions; see \cite{jyz} for some
other examples.

In Section \ref{s3}, we introduce the spaces $\vmo$ and $\twv$ and
give some basic properties of these spaces. In particular, we
characterize the space $\mathrm{VMO}_{\ro,\,L}(\rn)$ via the space
$T^{\fz}_{\oz,\mathrm v}({\rr}^{n+1}_+)$; see Theorem \ref{t3.2}
below. As an application of Theorem \ref{t3.2} together with a characterization
of the space $\lz_{n(1/p-1)}(\rn)$ in \cite{wws}, we obtain that
when $L=\Delta$, $p\in (0,1]$ and $\oz(t)=t^p$ for all $t\in (0,\fz)$, the space
$\vmo$ coincides with the space $\lz_{n(1/p-1)}(\rn)$; see Corollary \ref{c3.1} below.

In Section \ref{s4}, we introduce the space $\wtw$,
which is defined to be the Banach completion of the tent
space $\twl$ (see Definition \ref{dp4.1} below),
and prove that $\wtw$ is the dual space of $\twv$;
see Theorem \ref{t4.2} below. If $p\in (0,1]$ and
$\oz(t)=t^p$ for all $t\in (0,\fz)$,
then the tent space $\twl=T_2^p(\rnz)$ and
its predual space was proved to be the corresponding
space $\twv$ by Wang in \cite{wws}, which when $p=1$
plays a key role in \cite{ddsty}.
To prove that $\wtw$ is the dual space of $\twv$, different
from the approach used in \cite{cw} and \cite{wws} which strongly
depends on an unfamiliar result (Exercise 41 on \cite[p.\,439]{ds})
from the functional analysis, we only use the basic fact that the dual space
of $L^2$ is itself. Indeed, using this fact, for any $\ell \in (\twv)^\ast$,
we construct $g\in \wtw$ such that for all $f\in \twv$,
$$\ell(f)=\int_{\rnz}f(x,t)g(x,t)\frac{\,dx\,dt}{t};$$
see Theorem \ref{t4.2} below. As an application of Theorem \ref{t4.2}, we
further prove that the dual space of the space
$\vmoz$ is the space $B_{\oz,L}(\rn)$, where
the space $B_{\oz,L}(\rn)$ is the Banach completion of $\hw$; see
Definition \ref{dp4.3} and Theorem
\ref{t4.4} below. Since all dual spaces are complete, it is necessary here to
replace the Orlicz-Hardy space by its Banach completion, which is different
from \cite{ddsty}. In \cite{ddsty}, the Banach completion
of the Hardy space $H_L^1(\rn)$ is just itself.
Finally, in Subsection \ref{s4.3}, we give several examples
of operators to which the results of this paper are
applicable.

Let us make some conventions. Throughout the paper, we denote
by $C$ a positive constant which is independent of the main
parameters, but it may vary from line to line. The symbol $X \ls Y$
means that there exists a positive constant $C$ such that $X \le
CY$; the symbol $\lfr\az\rf$ for $\az\in\rr$ denotes the maximal
integer no more than $\az$; $B\equiv B(z_B,\,r_B)$ denotes an open ball with
center $z_B$ and radius $r_B$ and $CB(z_B,\,r_B)\equiv
B(z_B,\,Cr_B).$ Set $\cn\equiv\{1,2,\cdots\}$ and
$\zz_+\equiv\cn\cup\{0\}.$ For any subset $E$ of $\rn$, we denote by
$E^\com$ the set $\rn\setminus E.$

\section{Preliminaries\label{s2}}

\hskip\parindent In this section, we first
describe some basic assumptions on the operators $L$ and
Orlicz functions studied in this paper (see, for
example, \cite{dm, m, ja, v, dxy,
dy1, dy2, ddsty, jyz}), and we then recall some notions about
the Orlicz-Hardy space $H_{\oz,L}(\rn)$ and the BMO-type space
$\bmor$ in \cite{jyz}.

\subsection{Two assumptions on the operator $L$\label{s2.1}}

\hskip\parindent Let $\nu\in(0,\pi)$, $S_\nu\equiv\{z\in
\cc:\, |\arg(z)|\le \nu\}\cup \{0\}$ and $S^0_\nu$ the
interior of $S_\nu$, where
$\arg(z)\in(-\pi,\,\pi]$ is the argument of $z$.
Assume that $L$  is a linear operator such that
$\sigma(L)\subset S_\nu$, where $\sz(L)$
denotes the spectra of $L$ and $\nu\in (0,\pi/2)$, and that for all $\gz >\nu$,
there exists a positive constant $C_\gz$ such that
$$\|(L-\lz I)^{-1}\|_{L^2(\rn)\to L^2(\rn)}\le C_\gz|\lz|^{-1}, \ \ \
\forall \ \lz \notin S_\gz,$$
where and in what follows, for any two normed linear spaces
$\mathscr{X}$ and $\mathscr{Y}$, and any bounded linear operator
$T$ from $\mathscr{X}$ to $\mathscr{Y},$ we use
$\|T\|_{\mathscr{X}\to\mathscr{Y}}$ to denote the operator norm of $T$
from $\mathscr{X}$ to $\mathscr{Y}$ and $\cl({\mathscr{X},\mathscr{Y}})$
the set of all bounded linear operators from $\mathscr{X}$ to
$\mathscr{Y}$. Hence $L$ generates a holomorphic semigroup
$e^{-zL}$, where $0\le |\arg(z)| <\frac{\pi}{2}-\nu$ (see \cite{m}).
We make the following two assumptions on $L$ (see \cite{dy2,
dxy, ddsty, ya2,jyz}).

\begin{proof}[\bf Assumption (a)]\rm Assume that for all $t > 0$, the
distribution kernels $p_t$ of $e^{-tL}$ belong to
$L^\fz(\rn\times\rn)$ and satisfy the estimate $|p_t(x,y)|\le h_t(x,y)$
for all $x,\, y\in \rn$, where $h_t$ is given by
\begin{equation}\label{2.1}
h_t(x,\,y)=t^{-\frac{n}{m}}g\lf(\frac{|x-y|}{t^{\frac{1}{m}}}\r),
\end{equation}
in which $m$ is a positive constant and $g$ is a positive, bounded,
decreasing function satisfying that
\begin{equation}\label{2.2}
\lim_{r\to \fz}r^{n+\ez}g(r)=0
\end{equation}
for some $\ez > 0$.
\end{proof}

Let $H(S^0_\nu )$ be the space of all holomorphic
functions on $S^0_\nu$ and
$$H_\fz(S^0_\gz)\equiv \lf\{b\in H(S^0_\gz):\,\|b\|_\fz\equiv
\sup_{z\in S^0_\gz}|b(z)|<\fz\r\}.$$
Recall that the operator $L$ is said to have a bounded
$H_\fz$-calculus in $L^2(\rn)$ (see \cite{m}) provided that for all
$\gz\in(\nu,\,\pi)$, there exists a positive constant $\wz C_\gz$
such that for all $b\in H_\fz(S^0_\gz)$, $b(L)\in\cl(L^2(\rn),\,
L^2(\rn))$ and
$\|b(L)\|_{L^2(\rn)\to L^2(\rn)}\le\wz C_\gz\|b\|_{\fz}$, where
$\psi(z)=z(1+z)^{-2}$ for all $z\in S^0_\gz$
and $b(L)\equiv [\psi(L)]^{-1}(b\psi)(L)$. It was proved in \cite{m} that
$b(L)$ is a well-defined linear operator in $L^2(\rn)$.

\begin{proof}[\bf Assumption (b)]\rm Assume that the
operator $L$ is one-to-one, has dense range in $L^2(\rn)$ and a
bounded $H_\fz$-calculus in $L^2(\rn).$
\end{proof}

From the assumptions (a) and (b), it is easy to deduce the
following useful estimates.

First, if $\{e^{-tL}\}_{t\ge 0}$ is a bounded analytic semigroup in
$L^2(\rn)$ whose kernels $\{p_t\}_{t\ge0}$ satisfy the estimates
\eqref{2.1} and \eqref{2.2}, then for any $k \in \cn$, there exists
a positive constant $C$ such that the time derivatives of $p_t$
satisfy that
\begin{equation} \label{2.3} \lf|t^k\frac{\partial^k
p_t(x,y)}{\partial t^k}\r| \le
\frac{C}{t^{\frac{n}{m}}}g\lf(\frac{|x-y|}{t^\frac{1}{m}}\r) \ \
\end{equation}
for all  $t>0$ and  almost everywhere $x,\,y\in \rn.$ It should be
pointed out that for any $k\in \cn$, the function $g$ may depend on
$k$ but it always satisfies $\eqref{2.2}$; see Theorem 6.17 of
\cite{ou} and \cite{cd}, and also \cite{dy2,
dxy, ddsty, ya2,jyz}.

Secondly, let
$$\Psi(S^0_\nu)\equiv\bigg\{\psi \in H(S^0_\nu):\,
 \exists\, s, C>0\ \mathrm{such \ that}\ \forall z\in S^0_\nu,
\,|\psi(z)|\le C|z|^s(1+|z|^{2s})^{-1} \bigg\}.$$
It is well known that $L$ has a bounded
$H_\fz$-calculus in $L^2(\rn)$ if and only if for all
$\gz\in(\nu,\,\pi]$ and any non-zero function $\psi \in
\Psi(S^0_\gz)$,  $L$ satisfies the square function estimate and its
reverse, namely, there exists a positive constant $C$ such that for
all $f\in L^2(\rn)$,
\begin{equation}\label{2.4}
C^{-1}\|f\|_{L^2(\rn)}\le
\lf(\int^\fz_0\|\psi_t(L)f\|^2_{L^2(\rn)}\frac{\,dt}{t}\r)^{1/2}\le
C\|f\|_{L^2(\rn)},
\end{equation}
where $\psi_t(\xi)=\psi(t\xi)$ for all $t>0$ and $\xi\in \rn.$
Notice that different choices of $\gz >\nu$  and $\psi\in
\Psi(S^0_\gz)$ lead to equivalent quadratic norms of $f$; see
\cite{m} for the details.

As noticed in \cite{m}, positive self-adjoint operators satisfy the
quadratic estimate \eqref{2.4}. So do normal operators with spectra
in a sector, and maximal accretive operators. For definitions of
these classes of operators, we refer the reader to \cite{yo}.

\subsection{An acting class of the
semigroup $\{e^{-tL}\}_{t\ge0}$\label{s2.2}}

\hskip\parindent Duong and Yan \cite{dy1} introduced the class of functions that
the operators $e^{-tL}$ act upon. Precisely, for any $\bz> 0$, let
$\cm_\bz(\rn)$ be the collection of all functions $f\in
L^2_{\loc}(\rn)$ such that
\begin{equation*}
\|f\|_{\cm_\bz(\rn)}\equiv\lf(\int_{\rn}\frac{|f(x)|^2}{1+|x|^{n+\bz}}\,dx\r)^{1/2}
<\fz.
\end{equation*}
Then $\cm_\bz(\rn)$ is a Banach space under the norm
$\|\cdot\|_{\cm_\bz(\rn)}.$ For any given operator $L$, set
\begin{equation}\label{2.5}
\tz(L)\equiv\sup\{\ez> 0:\  (\ref{2.2}) \ \rm  holds\}
\end{equation}
and define
\begin{equation*}
\begin{array}[C]{l}
\cm(\rn)\equiv{\lf\{
\begin{array}{ll}
\quad\cm_{\tz(L)}(\rn)\quad\quad &\quad\mbox{if}\quad\tz(L)<\fz;\\
\displaystyle\mathop\cup_{0<\bz<\fz}\cm_\bz(\rn) &\quad{\rm if }
\quad \tz(L)=\fz.
\end{array}
\r.}
\end{array}
\end{equation*}

Let $s\in \zz_+.$ For any $(x,\,t)\in
{\rr}^{n+1}_+\equiv\rn\times(0,\,\fz)$ and $f\in \cm(\rn)$, set
\begin{equation}\label{2.6} P_{s,\,t}f(x)\equiv
f(x)-(I-e^{-tL})^{s+1}f(x) \quad {\rm and} \quad Q_{s,\,t}f(x)\equiv
t^{s+1}L^{s+1}e^{-tL}f(x).
\end{equation}
If $s = 0$, write
\begin{equation}\label{2.7}
P_{t}f(x)\equiv P_{0,\,t}f(x)=e^{-tL}f(x) \quad {\rm and} \quad
Q_{t}f(x)\equiv Q_{0,\,t}f(x)=tLe^{-tL}f(x).
\end{equation}
For any $f\in\cm(\rn)$, by \eqref{2.3}, it is easy to show that
$P_{s,\,t}f$ and $Q_{s,\,t}f$ are well defined. Moreover, by
\eqref{2.3} again, we know that the kernels $p_{s,\,t}$ and $q_{s,\,t}$ of
$P_{s,\,t}$ of $Q_{s,\,t}$ satisfy that for all $t>0$ and $x,\,y\in \rn$,
\begin{equation}\label{2.8}
|p_{s,\,t^m}(x,y)|+|q_{s,\,t^m}(x,y)|\le C_st^{-n}g\lf(\frac{|x-y|}{t}\r),
\end{equation}
where the function $g$ satisfies the condition \eqref{2.2} and $C_s$ is
a positive constant independent of $t$, $x$ and $y$.

It should be pointed out that these operators in \eqref{2.6} were
introduced by Blunck and Kunstmann \cite{bk}.

\subsection{Orlicz functions \label{s2.3}}

\hskip\parindent Let $\omega$ be a positive function defined on
$\rr_+\equiv(0,\,\fz).$ The function $\omega$ is said to be of upper
type $p$ (resp. lower type $p$) for some $p\in[0,\,\fz)$, if there
exists a positive constant $C$ such that for all $t\geq 1$ (resp.
$0<t\le1$),
\begin{equation}\label{2.9}
\omega(st)\le Ct^p \omega(s).
\end{equation}

Obviously, if $\oz$ is of lower type $p$ for some $p>0$, then
$\lim_{t\to0^+}\oz(t)=0.$ So for the sake of convenience, if it is
necessary, we may assume that $\oz(0)=0.$ If $\oz$ is of both upper
type $p_1$ and lower type $p_0$, then $\oz$ is said to be of type
$(p_0,\,p_1).$ Let
\begin{equation*}
\wz p_1(\oz)\equiv\inf\{ p>0: (\ref{2.9}) \ \mathrm{holds\ for\
all}\ t\in(1,\fz)\},
\end{equation*}
and
\begin{equation*}
\wz p_0(\oz)\equiv\sup\{ p>0: (\ref{2.9}) \ \mathrm{holds\ for\
all}\ t\in(0,1)\}.\end{equation*}
It is easy to see that $\wz
p_0(\oz)\le\wz p_1(\oz)$ for all $\oz.$ In what follows, $\wz
p_0(\oz)$ and $\wz p_1(\oz)$ are called the critical lower
type index and the critical upper type index of $\oz$, respectively.
Throughout the whole paper, we always assume that $\oz$ satisfies
the following assumption.

\begin{proof}[\bf Assumption (c)]\rm Suppose that the positive Orlicz function
$\oz$ on $\rr_+$ is continuous, strictly increasing, subadditive, of
upper type $1$ and $\wz p_0(\oz)\in (\frac{n}{n+\tz(L)},1]$, where
$\tz(L)$ is as in \eqref{2.5}.
\end{proof}

Notice that for any $\oz$ of type $(p_0, p_1)$, if we set
$\wz\oz(t)\equiv\int_0^t\frac{\oz(s)}{s}\,ds$ for $t\in [0,\fz)$, then by
\cite[Proposition 3.1]{v}, $\wz\oz$ is equivalent to $\oz$, namely,
there exists a positive constant $C$ such that $C^{-1}\oz(t)\le
\wz\oz(t)\le C\oz(t)$ for all $t\in [0,\fz)$, and moreover, $\wz\oz$
is strictly increasing, subadditive and continuous function of type
$(p_0,\,p_1).$ Since all our results in this paper are invariant on equivalent
functions, we may always assume that $\oz$ satisfies the assumption (c);
otherwise, we may replace $\oz$ by $\wz\oz.$

We also make the following convention.

\begin{proof}[\bf Convention] From the assumption (c), it follows that
$\frac{n}{n+\tz(L)}<\wz p_0(\oz)\le \wz p_1(\oz)\le 1.$ In what
follows, if \eqref{2.9} holds for $\wz p_1(\oz)$ with $t\in
(1,\fz)$,
 then we choose $p_1(\oz)\equiv\wz p_1(\oz)$; otherwise  $ \wz p_1(\oz)<1$ and
 we choose $p_1(\oz)\in (\wz p_1(\oz),1).$ Similarly, if \eqref{2.9} holds for
$\wz p_0(\oz)$ with $t\in (0,1)$,
 then we choose $p_0(\oz)\equiv\wz p_0(\oz)$; otherwise
 we choose $p_0(\oz)\in (\frac{n}{n+\tz(L)}, \wz p_0(\oz))$ such that $\lfr
\frac nm(\frac{1}{p_0(\oz)}-1) \rf=\lfr\frac nm(\frac{1}{\wz
p_0(\oz)}-1) \rf$, where $m$ is as in \eqref{2.1}.
\end{proof}



Let $\oz$ satisfy the assumption (c). A measurable function $f$ on
$\rn$ is said to be in the Lebesgue type space $L(\oz)$ if
$\int_{\rn}\oz(|f(x)|)\,dx< \fz.$ Moreover, for any $f\in L(\oz)$, define
$$\|f\|_{L(\oz)}\equiv\inf\lf\{\lz>0:\ \int_{\rn}\oz\lf(\frac{|f(x)|}
{\lz}\r)\,dx\le 1\r\}.$$

Let $\oz$ satisfy the assumption (c). Define the function $\ro(t)$
on $\rr_+$ by setting, for all $t\in (0,\fz)$,
\begin{equation}\label{2.10}\ro(t)\equiv\frac{t^{-1}}{\oz^{-1}(t^{-1})},
\end{equation}
where $\oz^{-1}$ is the inverse function of
$\oz.$ Then by Proposition 2.1 in \cite{jyz}, $\ro$ is of type
$(1/p_1(\oz)-1,1/p_0(\oz)-1)$, which is denoted by
$(\bz_0(\ro),\bz_1(\ro))$ in what follows for short.

\subsection{The Orlicz-Hardy space $\hw$ and its dual space\label{s2.4}}

\hskip\parindent For any function $f\in L^1(\rn)$, the Lusin area function
$\cs_L(f)$ associated with the operator $L$ is defined by setting,
for all $x\in\rn$,
\begin{equation*}
\cs_L(f)(x)\equiv\lf(\int_{\Gamma(x)}|Q_{t^m}f(y)|^2\frac{\,dy\,dt}{t^{n+1}}\r)^{1/2},
\end{equation*}
where $Q_{t^m}$ is as in \eqref{2.7}. From the
assumption (b) together with \eqref{2.4}, it is easy to deduce that the
Lusin area function $\cs_L$ is bounded on $L^2(\rn).$ Auscher,
Duong and McIntosh \cite{adm2} proved that for any $p\in(1,\fz)$,
there exists a positive constant $C_p$ such that for all $f\in
L^p(\rn)$,\begin{equation}\label{2.11}
C_p^{-1}\|f\|_{L^p(\rn)}\le\|\cs_L(f)\|_{L^p(\rn)}\le
C_p\|f\|_{L^p(\rn)};
\end{equation}
see also Duong and McIntosh \cite{dm} and Yan \cite{ya1}. By
duality, the operator $S_{L^\ast} $ also satisfies the estimate
\eqref{2.11}, where $L^\ast$ is the adjoint operator of $L$ in
$L^2(\rn).$

Recall that the Orlicz-Hardy space $\hw$ and the BMO-type space
$\bmor$ were introduced in \cite{jyz}.

\begin{defn}\label{dp2.1}
Let $L$ satisfy the assumptions (a) and (b) and $\oz$ satisfy the
assumption (c). A function $f\in L^2(\rn)$ is said to be in $\wz
H_{\oz,\,L}(\rn)$ if $\cs_L(f)\in L(\oz)$, and moreover, define
$$\|f\|_{H_{\oz,\,L}(\rn)}\equiv \|\cs_L(f)\|_{L(\oz)}=\inf\lf
\{\lz>0:\int_{\rn}\oz\lf(\frac{\cs_L(f)(x)}{\lz}\r)\,dx\le 1\r\}.$$
The Orlicz-Hardy space $H_{\oz,\,L}(\rn)$ associated with the
operator $L$ is defined to be the  completion of $\wz
H_{\oz,\,L}(\rn)$  in the norm $\|\cdot\|_{H_{\oz,\,L}(\rn)}.$
\end{defn}

\begin{defn}\label{dp2.2}
Let $L$ satisfy the assumptions (a) and (b), $\oz$ satisfy the
assumption (c), $\ro$ be as in \eqref{2.10}, $q\in [1,\fz)$ and
$s\geq \lfr\frac nm(\frac{1}{\wz p_0(\oz)}-1)\rf.$ A function $f\in\cm(\rn)$ is
said to be in $\bmoq$ if
\begin{equation*}
\|f\|_{\bmoq}\equiv\sup_{B\subset\rn}\frac{1}{\ro(|B|)}\lf[\frac{1}{|B|}\int_B
|f(x)-P_{s,(r_B)^m}f(x)|^q \,dx\r]^{1/q}< \fz,
\end{equation*}
where the supremum is taken over all balls $B$ of $\rn.$
\end{defn}

\begin{rem}\label{r2.1}\rm
(i) Let $p\in (0,1]$, $q\in[1,\fz)$ and $s\ge \lfr\frac
nm(\frac 1p-1)\rf$. If $\oz(t)=t$ for all $t\in(0,\fz)$,
then $\hw=H_L^1(\rn)$ and $\bmoq=\mathrm{BMO}_L(\rn)$, where $H_L^1(\rn)$ and
$\mathrm{BMO}_L(\rn)$ were introduced by Duong and Yan
\cite{dy1,dy2}, respectively. If $p\in(n/(n+\tz(L)),1)$ and $\oz(t)=t^p$
for all $t\in(0,\fz)$, then $\hw=H_L^p(\rn)$ and $\bmoq=\cl_L(1/p-1,q,s)$, where
$H_L^p(\rn)$ and $\cl_L(1/p-1,q,s)$ were introduced by Yan \cite{ya2} and
Duong and Yan \cite{dy3}, respectively.

(ii) For $q\in[1,\fz)$ and $s\ge s_0\equiv\lfr \frac nm(\frac{1}{\wz
p_0(\oz)}-1)\rf$, the spaces $\bmoq$ coincide with
$\mathrm{BMO}_{\ro,L}^{2,s_0}(\rn)$; see Corollary 3.1 and Remark
4.4 of \cite{jyz}. Hence, in what follows, we denote the space
$\bmoq$ simply by $\bmor$.

(iii) It was proved in \cite{jyz} that the dual space of $H_{\oz,L}(\rn)$
 is the space  $\mathrm{BMO}_{\ro,L^\ast}(\rn)$, where $L^\ast$ denotes the adjoint
 operator of $L$ in $L^2(\rn)$.
\end{rem}

\section{$\vmo$-type spaces\label{s3}}

\hskip\parindent Suppose that the assumptions (a), (b) and (c) hold.
In this section, we study the spaces of functions with vanishing mean oscillation
 associated with operators and Orlicz functions. We begin with
 some notions and notation.

\begin{defn}\label{dp3.1} Let $L$ satisfy the assumptions (a) and (b),
$\oz$ satisfy the assumption (c), $\ro$ be as in \eqref{2.10}
and $s\geq \lfr\frac nm(\frac{1}{\wz p_0(\oz)}-1)\rf.$ A function
$f\in\bmor$ is said to be in $\vvmo$, if it satisfies the limiting
conditions $\gz_1(f)=\gz_2(f)=\gz_3(f)=0$, where
$$\gz_1(f)\equiv\lim_{c\rightarrow 0}\sup_{\mathrm{ball}\, B:\,r_B\le c}\bigg(\frac{1}{|B|
[\ro(|B|)]^2}\int_B|f(x)-P_{s,(r_B)^m}f(x)|^2\,dx\bigg)^{1/2},$$
$$\gz_2(f)\equiv\lim_{c\rightarrow \fz}\sup_{\mathrm{ball}\,B:\,r_B\ge c}\bigg(\frac{1}{|B|
[\ro(|B|)]^2}\int_B|f(x)-P_{s,(r_B)^m}f(x)|^2\,dx\bigg)^{1/2},$$
and
$$\quad\gz_3(f)\equiv\lim_{c\rightarrow \fz}\sup_{\mathrm{ball}\,B\subset [B(0,c)]^\com}
\bigg(\frac{1}{|B|
[\ro(|B|)]^2}\int_B|f(x)-P_{s,(r_B)^m}f(x)|^2\,dx\bigg)^{1/2}.$$ For any
function $f\in\vvmo$, we define $\|f\|_{\vvmo}\equiv\|f\|_{\bmor}$.
\end{defn}

We next present some properties of the space $\vvmo$. To this end, we
first recall some notions of tent spaces; see
\cite {cms,hsv,jyz}.

Let $\Gamma(x)\equiv\{(y,t)\in\rq:\,|x-y|<t\}$ denote the standard
cone (of aperture 1) with vertex $x\in\rn.$ For any closed set $F$
of $\rn$, denote by $\car{F}$ the union of all cones with vertices
in $F$, namely, $\car{F}\equiv\cup_{x\in F}\Gamma(x)$; and for any
open set $O$ in $\rn$, denote the tent over $O$ by $\widehat{O}$,
which is defined by $\widehat{O}\equiv[\car(O^\com)]^\com.$

For all measurable functions $g$ on ${\rr}^{n+1}_+$ and all
$x\in\rn$, define
\begin{equation*}
\ca(g)(x)\equiv
\lf(\int_{\Gamma(x)}|g(y,t)|^2\frac{\,dy\,dt}{t^{n+1}}\r)^{1/2}
\end{equation*}
and
\begin{equation*}
\ccc_{\ro}(g)(x)\equiv\sup_{\mathrm{ball}\, B\ni x}\frac{1}{\ro(|B|)}
\lf(\frac{1}{|B|}\int_{\widehat{B}}|g(y,t)|^2\frac{\,dy\,dt}{t}\r)^{1/2}.
\end{equation*}

For $p\in(0,\fz)$, the tent space $T^p_2(\rnz)$ is defined to be the set
of all measurable functions $g$ on $\rnz$ such that $\|g\|_{T^p_2(\rnz)}\equiv
\|\ca(g)\|_{L^p(\rn)}<\fz.$ The tent space $T_\oz({\rr}^{n+1}_+)$ associated to the
function $\oz$ is defined to be the set of all measurable functions
$g$ on $\rr^{n+1}_+$ such that $\ca(g)\in L(\oz)$, and its norm is
given by
$$\|g\|_{T_\oz({\rr}^{n+1}_+)}\equiv \|\ca(g)\|_{L(\oz)}=\inf
\lf\{\lz>0:\ \int_{\rn} \oz \lf(\frac{\ca(g)(x)}{\lz}\r) \,dx\le
1\r\};$$ the space $T^{\fz}_{\oz}({\rr}^{n+1}_+)$  is defined to be the
set of all measurable functions $g$ on $\rr^{n+1}_+$ such that
$\|g\|_{T^\fz_\oz({\rr}^{n+1}_+)}\equiv\|\ccc_{\ro}(g)\|_{L^\fz(\rn)}<\fz.$

In what follows, let $\tw0$ be the set of all $f\in\twz$ satisfying
$\eta_1(f)=\eta_2(f)=\eta_3(f)=0$, where
\begin{equation}\label{3.1}
\eta_1(f)\equiv \lim_{c\rightarrow0}\sup_{\mathrm{ball}\,B:\,r_B\le c}\frac{1}{\ro(|B|)}
\lf(\frac{1}{|B|}\int_{\widehat{B}}|f(y,t)|^2\frac{\,dy\,dt}{t}\r)^{1/2},
\end{equation}
$$\eta_2(f)\equiv \lim_{c\rightarrow\fz}\sup_{\mathrm{ball}\,B:\,r_B\ge c}\frac{1}{\ro(|B|)}
\lf(\frac{1}{|B|}\int_{\widehat{B}}|f(y,t)|^2\frac{\,dy\,dt}{t}\r)^{1/2},$$
and
$$\quad \eta_3(f)\equiv \lim_{c\rightarrow\fz}\sup_{\mathrm{ball}\,B\subset [B(0,c)]^\com}
\frac{1}{\ro(|B|)}
\lf(\frac{1}{|B|}\int_{\widehat{B}}|f(y,t)|^2\frac{\,dy\,dt}{t}\r)^{1/2}.$$
It is easy to see that $\tw0$ is a closed linear subspace of $\twz$.

Further, denote by $T^\fz_{\oz,1}(\rnz)$ the space of all $f\in\twz$
with $\eta_1(f)=0$, and $T^2_{2,c}(\rnz)$ the space of all $f\in
T^2_{2}(\rnz)$ with compact support. Obviously, we have
$T^2_{2,c}(\rnz)\subset \tw0\subset T^\fz_{\oz,1}(\rnz)$. Similarly,
let $\twc$ denote the set of all $f\in\twz$ with compact support.
It is easy to see that $\twc$ coincides with $T^2_{2,c}(\rnz)$.
Define $\twv$ to be the closure of $\twc$ in $T^\fz_{\oz,1}(\rnz)$.

Similarly to the proof of Lemma 3.2 in \cite{ddsty}, we have the
following proposition. We omit the details.

\begin{prop}\label{p3.1}
Let $\twv$ and $\tw0$ be defined as above. Then $\twv$ and $\tw0$
coincide with equivalent norms.
\end{prop}

Recall that a measure $d\mu$ on ${\rr}^{n+1}_+$ is said to be a
$\ro$-Carleson measure if
\begin{equation*}
\sup_{B\subset
\rn}\lf\{\frac{1}{|B|[\ro(|B|)]^2}\int_{\widehat{B}}|d\mu|\r\}<\fz,
\end{equation*}
where the supremum is taken over all balls $B$ of $\rn$; see \cite{hsv,jyz}.

Let $m$ be the constant in \eqref{2.1}, $s\ge s_1\ge \lfr \frac
nm(\frac{1}{\wz p_0(\oz)}-1)\rf$, where $\wz p_0(\oz)$ is the
critical lower type index of $\oz$. Let $C_{m,s,s_1}$ be a positive
constant such that
\begin{equation}\label{3.2}
C_{m,s,s_1}\int^\fz_0t^{m(s+2)}e^{-2t^m}(1-e^{-t^m})^{s_1+1}\frac{\,dt}{t}=1.
\end{equation}

The following Lemma \ref{l3.1} and Theorem \ref{t3.1} were established in \cite{jyz}.

\begin{lem}\label{l3.1}
Let $L$, $\oz$ and $\rho$ be as in Definition
\ref{dp3.1}, and $s\ge s_1\ge \lfr \frac nm(\frac{1}{\wz p_0(\oz)}-1)\rf$.
Suppose that  $f\in\cm(\rn)$ such that
$|Q_{s,\,t^m}(I-P_{s_1,t^m})f(x)|^2\,dx\,dt/t$ is a $\ro$-Carleson
measure on $\rnz$ and $g\in\hz\cap L^2(\rn)$. Then
\begin{equation*}
\int_{\rn}f(x)g(x)\,dx=C_{m,s,s_1}\int_{\rnz}
Q_{s,t^m}(I-P_{s_1,t^m})f(x)Q^\ast_{t^m}g(x)\,\frac{dx\,dt}{t}.
\end{equation*}
\end{lem}

\begin{thm}\label{t3.1}
Let $L$, $\oz$ and $\ro$ be as in Definition \ref{dp3.1}, and $s\ge s_1\ge
\lfr\frac nm(\frac{1}{\wz p_0(\oz)}-1)\rf$.
Then the following conditions are equivalent:

(a) $f\in \bmor$;

(b) $f\in \cm(\rn)$ and $|Q_{s,\,t^m}(I-P_{s_1,t^m})f(x)|^2\,dx\,dt/t$ is
a $\ro$-Carleson measure  on ${\rr}^{n+1}_+$.\\
Moreover, $\|Q_{s,\,t^m}(I-P_{s_1,t^m})f\|_{\twz}$ is equivalent to
$\|f\|_{\bmor}.$
\end{thm}

We now establish a characterization of the space $\vvmo$ via the space $\twv$
by borrowing some ideas from \cite{ddsty}. We remark that
if $\oz(t)=t$ for all $t\in (0,\fz)$, then Theorem
\ref{t3.2} coincides with Proposition 3.3 in \cite{ddsty}.

\begin{thm}\label{t3.2}
Let $L$, $\oz$ and $\ro$ be as in Definition \ref{dp3.1}, $s_1\ge
s_0\ge\lfr\frac nm(\frac{1}{\wz p_0(\oz)}-1)\rf$ and $s\ge 2s_1$.
Then the following conditions are equivalent:

(a) $f\in \vmor$;

(b) $f\in \cm(\rn)$ and $Q_{s,\,t^m}(I-P_{2s_1,t^m})f\in \twv$.
\end{thm}
\begin{proof}\rm We first notice that by the assumption (c),
there exists $\ez\in (n\bz_1(\ro), \tz(L)).$
Recall that $\bz_1(\rho)=1/p_0(\oz)-1$ and $p_0(\oz)$
is as in the convention.
To see that (a) implies (b), by the fact $\vmor \subset\bmor$
together with Theorem \ref{t3.1}, we only need to verify
$Q_{s,\,t^m}(I-P_{2s_1,t^m})f\in \twv$. To this end, we need
to show that for all balls $B\equiv B(x_B,r_B)$,
\begin{equation}\label{3.3}
\frac{1}{\ro(|B|)|B|^{1/2}}\bigg(\int_{\widehat
B}|Q_{s,\,t^m}(I-P_{2s_1,t^m})f(x)|^2\frac{\,dx\,dt}{t}\bigg)^{1/2}\ls
\sum_{k=1}^\fz 2^{-k\ez}\dz_k(f,B),
\end{equation}
where $\ez\in(n\bz_1(\ro),\tz(L))$ and
\begin{equation*}
\dz_k(f,B)\equiv \sup_{B'\subset
2^{k+1}B:\,r_{B'}\in[r_B/2,2r_B]}\frac{1}{\ro(|B'|)|B'|^{1/2}}
\bigg(\int_{B'}|f(x)-P_{s_0,(r_{B'})^m}f(x)|^2\,dx\bigg)^{1/2}.
\end{equation*}
In fact, since $f\in\vmor\subset\bmor$, we have $\dz_k(f,B)\le
\|f\|_{\bmor}$. Moreover, for each $k\in\cn$, we have
\begin{eqnarray*}
 \lim_{c\rz 0}\sup_{B:\,r_B\le c}\dz_k(f,B)&&=\lim_{c\rz
 \fz}\sup_{B:\,r_B\ge c}\dz_k(f,B)=\lim_{c\rz \fz}\sup_{B\subset
 B(0,\,c)^\com}\dz_k(f,B)=0.
\end{eqnarray*}
Then by \eqref{3.3} and the dominated convergence
theorem for series, we have
\begin{eqnarray}\label{3.4}
&&\eta_1(Q_{s,\,t^m}(I-P_{2s_1,t^m})f)\\
&&\hs=\lim_{c\rz 0}\sup_{B:\,r_B\le c}
\frac{1}{\ro(|B|)|B|^{1/2}}\bigg(\int_{\widehat
B}|Q_{s,\,t^m}(I-P_{2s_1,t^m})f(x)|^2
\frac{\,dx\,dt}{t}\bigg)^{1/2}\nonumber\\
&&\hs\ls \sum_{k=1}^{\fz}2^{-k\ez}\lim_{c\rz
0}\sup_{B:\,r_B\le c}\dz_k(f,B)=0;\nonumber
\end{eqnarray}
thus, $\eta_1(Q_{s,\,t^m}(I-P_{2s_1,t^m})f)=0$. Similarly, we have
$\eta_2(Q_{s,\,t^m}(I-P_{2s_1,t^m})f)=0=\eta_3(Q_{s,\,t^m}(I-P_{2s_1,t^m})f)$,
which implies that $Q_{s,\,t^m}(I-P_{2s_1,t^m})f\in \twv$, and hence, (a) implies (b).

Let us now prove \eqref{3.3}. Notice that
\begin{eqnarray*}
&&\frac{1}{\ro(|B|)|B|^{1/2}}\bigg(\int_{\widehat
B}|Q_{s,\,t^m}(I-P_{2s_1,t^m})f(x)|^2\frac{\,dx\,dt}{t}\bigg)^{1/2}\\&&\quad\le
\frac{1}{\ro(|B|)|B|^{1/2}}\bigg(\int_{\widehat
B}|Q_{s,\,t^m}(I-P_{2s_1,t^m})(I-P_{s_0,\,(r_{2B})^m})f(x)|^2\frac{\,dx\,dt}{t}\bigg)^{1/2}
\\&&\quad\quad+\frac{1}{\ro(|B|)|B|^{1/2}}\bigg(\int_{\widehat
B}|Q_{s,\,t^m}(I-P_{2s_1,t^m})P_{s_0,\,(r_{2B})^m}f(x)|^2\frac{\,dx\,dt}{t}\bigg)^{1/2}\\
&&\quad \equiv \mathrm{I}+\mathrm{J}.
\end{eqnarray*}
Set $U_1(B)\equiv 2B$ and $U_k(B)\equiv 2^kB\backslash2^{k-1}B$ when $k\ge 2$.
For $k\in\cn$, let $b_k\equiv [(I-P_{s_0,\,(r_{2B})^m})f]\chi_{U_k(B)}$.
Thus, for $\mathrm{I}$, we have
\begin{eqnarray*}
  \mathrm{I}&&\le \sum_{k=1}^\fz\frac{1}{\ro(|B|)|B|^{1/2}}\bigg(\int_{\widehat
B}|Q_{s,\,t^m}(I-P_{2s_1,t^m})b_k(x)|^2\frac{\,dx\,dt}{t}\bigg)^{1/2}\equiv
\sum_{k=1}^\fz \mathrm{I}_k.
\end{eqnarray*}
When $k=1$, by \eqref{2.4}, we have
\begin{equation}\label{3.5}
\mathrm{I_1}\ls
\frac{1}{\ro(|B|)|B|^{1/2}}\|b_1\|_{L^2(\rn)}\ls
\dz_2(f,B).\end{equation} When $k\ge2$, notice that for any $x\in B$
and $y\in (2^kB)^\com$, $|x-y|\gs 2^kr_B$. Thus, by \eqref{2.3}, we
obtain
\begin{eqnarray*}|Q_{s,\,t^m}(I-P_{2s_1,t^m})b_k(x)|&&\ls \int_{U_k(B)}
\frac{t^\ez}{(t+|x-y|)^{n+\ez}}|(I-P_{s_0,(r_{2B})^m})f(y)|\,dy\\
&&\ls t^\ez
(2^kr_B)^{-n-\ez}\int_{2^kB}|(I-P_{s_0,(r_{2B})^m})f(y)|\,dy.
\end{eqnarray*}
To estimate the last term, by the argument in
\cite[pp.\,645-646]{ddy}, we have that for any ball $B(x_B, 2^kr_B)$
with $k\ge2$, there exists a collection $\{B_{k,1},B_{k,2},\cdots, B_{k,N_k}\}$
of balls such that each ball $B_{k,i}$ is
of radius $r_{2B}$, $B(x_B, 2^kr_B)\subset\cup^{N_k}_{i=1} B_{k,i}$
and $N_k \ls 2^{nk}.$ From this facts and the H\"older inequality, it follows that
\begin{eqnarray*}|Q_{s,\,t^m}(I-P_{2s_1,t^m})b_k(x)|&&\ls \sum_{i=1}^{N_k}t^\ez
(2^kr_B)^{-n-\ez}\int_{B_{k,i}}|(I-P_{s_0,(r_{2B})^m})f(y)|\,dy\\&&\ls
\bigg(\frac{t}{r_B}\bigg)^\ez 2^{-k\ez}\ro(|B|)\dz_k(f,B),
\end{eqnarray*}
which yields \begin{equation}\label{3.6}
\mathrm{I}_k\ls
\frac{1}{|B|^{1/2}} \lf(\int_{\widehat
B}t^{2\ez-1}\,dx\,dt\r)^{1/2} 2^{-k\ez}(r_B)^{-\ez}\dz_k(f,B)\ls
2^{-k\ez}\dz_k(f,B).\end{equation} Combining \eqref{3.5} and
\eqref{3.6}, we obtain $\mathrm{I}\ls\sum_{k=1}^\fz
2^{-k\ez}\dz_k(f,B).$

 To estimate the term $\mathrm J$, write $t_{2B}\equiv (r_{2B})^m$ and
\begin{eqnarray*}
  &&I-P_{2s_1,t^m}\\&&\quad=(I-e^{-t_{2B}L}+e^{-t_{2B}L}-e^{-t^mL})^{2s_1+1}\\
  &&\quad=\bigg(\sum_{k=0}^{s_1}C_{2s_1+1}^k
  (-1)^ke^{-kt^mL}(I-P_{2s_1-s_0-k-1,t_{2B}})(I-P_{k-1,t_{2B}-t^m})
  \bigg)(I-P_{s_0,t_{2B}})
  \\&&\quad\quad+\bigg(\sum_{k=s_1+1}^{2s_1}C_{2s_1+1}^k
  (-1)^ke^{-kt^mL}(I-P_{2s_1-k,t_{2B}})(I-P_{k-s_0-2,t_{2B}-t^m})\bigg)
  (I-P_{s_0,t_{2B}-t^m})\\
  &&\quad\equiv
  \Psi_1(L)(I-P_{s_0,t_{2B}})+\Psi_2(L)(I-P_{s_0,t_{2B}-t^m}),
\end{eqnarray*}
where $t\in(0,r_B)$ and $C_{2s_1+1}^k$ denotes the combinatorial
number. Further, by \eqref{2.8}, we see that $k_{t,r_B}$, the kernel
of $Q_{s,t^m}P_{s_0,(r_{2B})^m}\Psi_1(L)$, satisfies that for all
$x,\,y\in\rn$,
\begin{equation*}
|k_{t,r_B}(x,y)|\ls
\bigg(\frac{t}{r_B}\bigg)^{m(s+1)}\frac{(r_B)^\ez}{(r_B+|x-y|)^{n+\ez}}.
\end{equation*}
By this and some computation similar to the estimate for $\mathrm{I}_k$,
we obtain that for all $x\in B$,
\begin{eqnarray*}
&&|Q_{s,\,t^m}P_{s_0,\,(r_{2B})^m}\Psi_1(L)(I-P_{s_0,(r_{2B})^m})f(x)|\\
&&\quad\ls\sum_{k=1}^\fz |Q_{s,\,t^m}P_{s_0,\,(r_{2B})^m}\Psi_1(L)
\lf[\chi_{U_k(B)}(I-P_{s_0,(r_{2B})^m})f\r](x)|\\
&&\quad\ls\bigg(\frac{t}{r_B}\bigg)^{m(s+1)}\sum_{k=1}^\fz2^{-k\ez}\ro(|B|)\dz_k(f,B).
\end{eqnarray*}
Similarly, we have that for all $x\in B$,
$$|Q_{s,\,t^m}P_{s_0,\,(r_{2B})^m}\Psi_2(L)(I-P_{s_0,(r_{2B})^m-t^m})f(x)|\ls
\bigg(\frac{t}{r_B}\bigg)^{m(s+1)}\sum_{k=1}^\fz2^{-k\ez}\ro(|B|)\dz_k(f,B).$$
The above two estimates yield
\begin{equation*}\mathrm{J}\ls
\frac{1}{|B|^{1/2}}\bigg[\int_{\widehat
B}\bigg(\frac{t}{r_B}\bigg)^{2m(s+1)}\frac{\,dx\,dt}{t}\bigg]^{1/2}
\sum_{k=1}^\fz2^{-k\ez}\dz_k(f,B)\ls\sum_{k=1}^\fz2^{-k\ez}\dz_k(f,B).
\end{equation*}
The estimates for $\mathrm I$ and $\mathrm J$ give \eqref{3.3}. Thus (a) implies (b).

Conversely, if (b) holds, by Theorem \ref{t3.1}, we see that $f\in \bmor$.
Notice that
\begin{eqnarray*} &&\frac{1}{\ro(|B|)|B|^{1/2}}
\bigg(\int_{B}|f(x)-P_{s_0,(r_B)^m}f(x)|^2\,dx\bigg)^{1/2}\\
&&\quad\quad\quad\quad=\sup_{\|g\|_{L^2(B)}\le
1}\frac{1}{\ro(|B|)|B|^{1/2}}\bigg|\int_B
f(x)(I-P^\ast_{s_0,(r_B)^m})g(x)\,dx\bigg|\nonumber.
\end{eqnarray*}
Then by Lemma \ref{l3.1}, we obtain
\begin{eqnarray*}&&\bigg|\int_B
f(x)(I-P^\ast_{s_0,(r_B)^m})g(x)\,dx\bigg|\\&&\quad
=\bigg|C_{m,s,2s_1}\int_{\rnz}Q_{s,t^m}(I-P_{2s_1,t^m})
f(x)Q^\ast_{t^m}(I-P^\ast_{s_0,(r_B)^m})g(x)\frac{\,dx\,dt}{t}\bigg|\\
&&\quad \ls
\int_{\widehat{4B}}|Q_{s,t^m}(I-P_{2s_1,t^m})f(x)Q^\ast_{t^m}
(I-P^\ast_{s_0,(r_B)^m})g(x)|
\frac{\,dx\,dt}{t}+\sum_{k=2}^\fz\int_{\widehat{2^{k+1}B}
\backslash\widehat{2^{k}B}}\cdots\\
&&\quad \equiv \mathrm{A}_1+\sum_{k=2}^\fz \mathrm{A}_k.
\end{eqnarray*}
Then the H\"older inequality and \eqref{2.4} yield
\begin{eqnarray*}
  \mathrm{A}_1&&\le
  \bigg(\int_{\widehat{4B}}|Q_{s,t^m}(I-P_{2s_1,t^m})f(x)|^2
  \frac{\,dx\,dt}{t}\bigg)^{1/2}
\bigg(\int_{\widehat{4B}}|Q^\ast_{t^m}(I-P^\ast_{s_0,(r_B)^m})
g(x)|^2\frac{\,dx\,dt}{t}\bigg)^{1/2}\\
&&\ls\bigg(\int_{\widehat{4B}}|Q_{s,t^m}(I-P_{2s_1,t^m})f(x)|^2
\frac{\,dx\,dt}{t}\bigg)^{1/2}.
\end{eqnarray*}
When $k\ge2$, notice that
\begin{eqnarray*}&&Q^\ast_{t^m}(I-P^\ast_{s_0,(r_B)^m})g=t^mL^\ast
e^{-t^mL^\ast}(I-e^{-(r_B)^mL^\ast})^{s_0+1}g\\
&&=\int_0^{(r_B)^m}\cdots\int_0^{(r_B)^m}\frac{t^m}{(t^m+r_1+\cdots+r_{s_0+1})^{s_0+2}}
Q^\ast_{s_0+1,t^m+r_1+\cdots+r_{s_0+1}}g\,dr_1\cdots dr_{s_0+1}.
\end{eqnarray*}
From \eqref{2.3}, it is easy to deduce that
$q_{s_0+1,t^m+r_1+\cdots+r_{s_0+1}}$, the kernel of
$Q^\ast_{s_0+1,t^m+r_1+\cdots+r_{s_0+1}}$, satisfies that for all
$x,\,y\in\rn$,
$$|q_{s_0+1,t^m+r_1+\cdots+r_{s_0+1}}(x,y)|\ls\frac{(t^m+r_1+
\cdots+r_{s_0+1})^{\ez/m}}{(t+|x-y|)^{n+\ez}},$$ where $\ez\in
(n\bz_1(\ro), \tz(L))$. Since $(x,t)\in
\widehat{2^{k+1}B}\backslash\widehat{2^{k}B}$ and $y\in B$, we have
$|x-y|+t\sim 2^kr_B$. Moreover, notice that $\ez>n\bz_1(\ro)$, which implies
that $\ez'\equiv(\ez-n\bz_1(\ro))/2>0$. By the arithmetic-geometric
inequality and the H\"older inequality, we have that for all $(x,t)\in
\widehat{2^{k+1}B}\backslash\widehat{2^{k}B}$,
\begin{eqnarray*}
&&|Q^\ast_{t^m}(I-P^\ast_{s_0,(r_B)^m})g(x)|\\
&&\hs\ls
\int_0^{(r_B)^m}\cdots\int_0^{(r_B)^m}\int_B\frac{t^m(t^m+r_1+
\cdots+r_{s_0+1})^{\ez/m}}{(t^m+r_1+\cdots+r_{s_0+1})^{s_0+2}(t+|x-y|)^{n+\ez}}\\
&&\hs\hs\times|g(y)|dy\,dr_1\cdots
dr_{s_0+1}\\
&&\hs\ls \frac{\|g\|_{L^1(B)}}{(2^kr_B)^{n+\ez}}\int_0^{(r_B)^m}\cdots
\int_0^{(r_B)^m}t^{m-m(1-\ez'/m)}
(r_1\cdots r_{s_0+1})^{-1+\frac{\ez-\ez'}{m(s_0+1)}} \,dr_1\cdots
dr_{s_0+1}\\
&&\hs\ls (2^kr_B)^{-n-\ez}|B|^{1/2}t^{\ez'}(r_B)^{\ez-\ez'},
\end{eqnarray*}
which implies that
\begin{eqnarray*}
&&\bigg(\int_{\widehat{2^{k+1}B}\backslash\widehat{2^{k}B}}|Q^\ast_{t^m}
(I-P^\ast_{s_0,(r_B)^m})g(x)|^2\frac{\,dx\,dt}{t}\bigg)^{1/2}\ls
2^{-k(n/2+\ez-\ez')}.
\end{eqnarray*}
From this together with the H\"older inequality, it follows that
\begin{eqnarray*}\mathrm{A}_k&&\ls
 2^{-k(n/2+\ez-\ez')}\bigg(\int_{\widehat{2^{k+1}B}\backslash
 \widehat{2^{k}B}}|Q_{s,t^m}
  (I-P_{2s_1,t^m})f(x)|^2\frac{\,dx\,dt}{t}\bigg)^{1/2}\\
  &&\ls 2^{-k(\ez-n\bz_1(\ro))/2}\frac{|B|^{1/2}\ro(|B|)}{|2^kB|^{1/2}\ro(|2^kB|)}
\bigg(\int_{\widehat{2^{k+1}B}\backslash\widehat{2^{k}B}}|Q_{s,t^m}
  (I-P_{2s_1,t^m})f(x)|^2\frac{\,dx\,dt}{t}\bigg)^{1/2}.
\end{eqnarray*}
Combine the estimates of $\mathrm{A}_k$, we finally obtain that
\begin{eqnarray*}
&&\frac{1}{\ro(|B|)|B|^{1/2}}
\bigg(\int_{B}|f(x)-P_{s_0,(r_B)^m}f(x)|^2\,dx\bigg)^{1/2}\ls
\sum_{k=1}^\fz 2^{-k(\ez-n\bz_1(\ro))/2}\sz_k(f,B),
\end{eqnarray*}
where
$$\sz_k(f,B)\equiv\frac{1}{|2^kB|^{1/2}\ro(|2^kB|)}
\bigg(\int_{\widehat{2^{k+1}B}\backslash\widehat{2^{k}B}}|Q_{s,t^m}
  (I-P_{2s_1,t^m})f(x)|^2\frac{\,dx\,dt}{t}\bigg)^{1/2}.$$
Since $Q_{s,t^m}(I-P_{2s_1,t^m})f\in \twv$, by Proposition \ref{p3.1},
for each $k\in \cn$, we have
\begin{eqnarray*}
 \lim_{c\rz 0}\sup_{B:\,r_B\le c}\sz_k(f,B)&&=\lim_{c\rz
 \fz}\sup_{B:\,r_B\ge c}\sz_k(f,B)=\lim_{c\rz \fz}\sup_{B\subset
 B(0,\,c)^\com}\sz_k(f,B)=0.
\end{eqnarray*}
Then similarly to the proof of \eqref{3.4}, we obtain that
$\gz_1(f)=\gz_2(f)=\gz_3(f)=0$, which implies that
$f\in \vmor$. This finishes the proof of Theorem \ref{t3.2}.
\end{proof}

\begin{rem}\label{r3.1}
\rm From Theorem \ref{t3.2}, it is easy to deduce
that if $s\ge s_0\ge\lfr\frac nm(\frac{1}{\wz p_0(\oz)}-1)\rf$, then
$\vmor=\vvmo$. Hence, in what follows, we denote the space
$\vmor$ simply by $\vmo.$
\end{rem}

Now, let $L=\Delta$ be the Laplacian operator on $\rn$, $p\in (0,1]$ and
$\oz(t)=t^p$ for all $t\in (0,\fz)$.
It was proved in \cite{jyz} that the space $\bmor$ coincides
with the space $\Lambda_{n(1/p-1)}(\rn)$. We now show that the
space $\vmo$ and the space $\lz_{n(1/p-1)}(\rn)$ coincide.

Recall that the space $\lz_{n(1/p-1)}(\rn)$ is defined to be the
closure of $\cs(\rn)$ (the space of Schwartz functions) in the norm
of the Lipschitz space $\Lambda_{n(1/p-1)}(\rn)$; see \cite{ja,pe,wws}.
Let $\psi\in\cs(\rn)$ and $\int_{\rn} \psi(x)x^\nu\,dx=0$ for all
$\nu\in \zz_+^{n}$, $|\nu|\le \lfr n(1/p-1)\rf$ and
$\psi_t(x)=t^{-n}\psi(x/t)$ for all $x\in\rn$ and $t\in(0,\fz)$.
Wang \cite{wws} established the following characterization of the space
$\lz_{n(1/p-1)}(\rn)$.
\begin{thm}\label{t3.3}
Let $p\in(0,1]$ and $\psi$ be as above. A function
$f\in \lz_{n(1/p-1)}(\rn)$ if and only if $f*\psi_t\in
\twz$ and $\eta_1(f*\psi_t)=0$, where $\eta_1(f*\psi_t)$
is defined as in \eqref{3.1}.
\end{thm}

The following result is deduced from Theorems \ref{t3.2} and
\ref{t3.3}.
\begin{cor}\label{c3.1}
Let $L=\Delta$, $p\in (0,1]$ and $\oz(t)=t^p$ for all  $t\in (0,\fz)$.
Then the space $\vmo$ coincides with the space $\lz_{n(1/p-1)}(\rn)$ with
equivalent norms.
\end{cor}

\begin{proof}\rm
Since $L=\Delta$, we then have $m=2$ and $\tz(L)=\fz$.
Let $s_1\ge \lfr \frac n2(\frac{1}{p}-1)\rf$,
$s\ge 2(s_1+1)$ and $f\in \vmo$. By Theorem \ref{t3.2}, we have
$Q_{s,\,t^m}(I-P_{2s_1,t^m})f\in \twv\subset \twz$, which together with
Proposition \ref{p3.1} further implies that
$$\eta_1(Q_{s,\,t^m}(I-P_{2s_1,t^m})f)=0.$$
Moreover, if we let $\ell\in\zz_+$ and $\cp_\ell$
denote the set of all polynomials with degree no more
than $\ell$, then $(I-P_{2s_1,t^m})(g)=0$ for all $g\in \cp_{2s_1}(\rn)$.
Thus, by Theorem \ref{t3.3}, we have $\vmo\subset \lz_{n(1/p-1)}(\rn)$.

Conversely, we first point out that it is easy to show that the space
$C_c^\fz(\rn)$ (the space of all $C^\fz(\rn)$ functions with
compact support) is dense in $\lz_{n(1/p-1)}(\rn)$. To prove
$\lz_{n(1/p-1)}(\rn)\subset \vmo$, by the completeness of these spaces, it suffices
to verify that $C_c^\fz(\rn)\subset \vmo$. Suppose that  $f\in
C_c^\fz(\rn)$. Then $f\in \Lambda_{n(1/p-1)}(\rn)=\bmor$. Thus, we only need to
show that $\gz_1(f)=\gz_2(f)=\gz_3(f)=0$.

Since $\tz(L)=\fz$, we can take $s\ge \lfr n(\frac{1}{p}-1)\rf$ and $\ez\in(s+1,\fz)$.
 Then for any $P_s\in\cps$ and any ball $B\equiv B(x_B,r_B)\subset \rn$, we have
$$\int_B |f(x)-P_{s,(r_B)^m}f(x)|^2\,dx=
\int_B |(I-P_{s,(r_B)^m})(f-P_s)(x)|^2\,dx.$$
For any $x\in\rn$, write
$$f(x)=\sum_{|\nu|\le s}\frac{D^\nu(f)(x_B)}{\nu!}(x-x_B)^\nu+\sum_{|\nu|=s+1}
\frac{D^\nu(f)(y_\nu)}{\nu!}(x-x_B)^\nu,$$
where for $\nu=(\nu_1,\cdots,\nu_n)\in(\zz_+)^n$,
$D^\nu=(\frac\partial{\partial x_1})^{\nu_1}\cdots
(\frac\partial{\partial x_n})^{\nu_n}$
and $y_\nu=\mu x+(1-\mu)x_B$ for certain $\mu\in (0,1)$.
Let
\begin{equation}\label{3.7}
P_s(x)=\sum_{|\nu|\le s}\frac{D^\nu(f)(x_B)}{\nu!}(x-x_B)^\nu.
\end{equation}
By \eqref{2.8}, $\ez>s+1$ and $f\in C_c^\fz(\rn)$, we have that
$$\int_B|f(x)-P_s(x)|^2\,dx\ls \sum_{|\nu|=s+1}
\int_B |(x-x_B)^\nu|^2\,dx\ls (r_B)^{2s+n+2},$$
and that for any $x\in B$,
\begin{eqnarray*}
  &&|P_{s,(r_B)^m}(f-P_s)(x)|\\
  &&\hs\ls (r_B)^{-n}\int_{2B}|(f-P_s)(y)|\,dy+
  \sum_{k=1}^\fz\int_{2^{k+1}B\backslash 2^kB} \frac{(r_B)^\ez}
  {(2^kr_B)^{n+\ez}}|(f-P_s)(y)|\,dy\\
  &&\hs\ls (r_B)^{s+1}+\sum_{k=1}^\fz \frac{(r_B)^\ez}{(2^kr_B)^{n+\ez}}
  \sum_{|\nu|=s+1}\int_{2^{k+1}B\backslash 2^kB}|(y-x_B)^\nu|\,dy\\
  &&\hs\ls (r_B)^{s+1}+\sum_{k=1}^\fz
  \frac{(r_B)^\ez}{(2^kr_B)^{n+\ez}} (2^{k+1}r_B)^{n+s+1}\ls (r_B)^{s+1},
   \end{eqnarray*}
which, together the fact that $s\ge \lfr n(\frac{1}{p}-1)\rf$, gives
$$\lim_{c\rightarrow 0}\sup_{B:\,r_B\le c}\frac{1}{|B|[\ro(|B|)]^2}\int_B
|f(x)-P_{s,(r_B)^m}f(x)|^2\,dx
\ls \lim_{r_B\rightarrow 0}\frac{1}{|B|^{2/p-1}}(r_B)^{2s+2+n}=0.$$
Thus $\gz_1(f)=0$.

To see $\gz_2(f)=0$, using $f\in C_c^\fz(\rn)$ and \eqref{2.8}, we
have that for any $t>0$,
$$\|P_{s,t^m}f\|_{L^2(\rn)}\ls \|f\|_{L^2(\rn)}.$$
Thus
\begin{eqnarray*}
\gz_2(f)&&=\lim_{c\rightarrow \fz}\sup_{B:\,r_B\ge c}
\frac{1}{|B|[\ro(|B|)]^2}\int_B|f(x)-P_{s,(r_B)^m}f(x)|^2\,dx\\
&&\ls \lim_{c\rightarrow \fz}\sup_{B:\,r_B\ge c}
\frac{1}{|B|[\ro(|B|)]^2}\|f\|^2_{L^2(\rn)}=0.
\end{eqnarray*}

Let us show that $\gz_3(f)=0$. By $\gz_2(f)=0$, we know that
for any $\bz>0$, there exists $R_0>0$ such that if $r_B\ge R_0$, then
$$\frac{1}{|B|[\ro(|B|)]^2}\int_B|f(x)-P_{s,(r_B)^m}f(x)|^2\,dx<\bz.$$
Thus, we only need to consider the case when $r_B\in (0,R_0)$. Since
$f\in C_c^\fz(\rn)$, then there exists $r_0>0$ such that
 $\supp f\subset B(0,r_0)$. Let $k_0\in\cn$ such that $2^{-k_0(\ez-s-1)}<\bz$.
 Taking $c>2^{k_0}(r_0+R_0)$, we deduce that for any ball
 $B\subset (B(0,c))^{\com}$ with $r_B<R_0$, $(2^{k_0}B)\cap B(0,r_0)=\emptyset.$
From this and \eqref{2.8}, we deduce that for any ball
 $B\subset (B(0,c))^{\com}$ with $r_B\in [1,R_0)$ and $x\in B$,
$$|P_{s,(r_B)^m}f(x)|\ls \int_{B(0,r_0)}\frac{(r_B)^\ez}{(2^{k_0}r_B)^{n+\ez}}
|f(y)|\,dy\ls (r_B)^{-n}2^{-k_0(n+\ez)}\|f\|_{L^1(\rn)},$$
which implies that
$$\frac{1}{|B|[\ro(|B|)]^2}\int_B|f(x)-P_{s,(r_B)^m}f(x)|^2\,dx\ls
\int_B|P_{s,(r_B)^m}f(x)|^2\,dx\ls 2^{-2k_0(n+\ez)}\ls \bz.$$

When $B\equiv B(x_B,r_B)\subset (B(0,c))^{\com}$ with $r_B\in(0,1)$,
by $\ez>s+1$ and $(2^{k_0}B)\cap B(0,r_0)=\emptyset$,
similarly to the proof of $\gz_1(f)=0$ with $P_s$ as in \eqref{3.7},
we have that for any $x\in B$,
\begin{eqnarray*}
|f(x)-P_{s,(r_B)^m}f(x)|&&=|(I-P_{s,(r_B)^m})(f-P_s)(x)|\\
&&\ls \sum_{k=k_0+1}^\fz
\sum_{|\nu|=s+1}\int_{2^kB\backslash 2^{k-1}B}\frac{(r_B)^\ez}
{(2^kr_B)^{n+\ez}}|(y-x_B)^\nu|\,dy\\
&&\ls \sum_{k=k_0+1}^\fz 2^{-k(n+\ez)}(r_B)^{-n}(2^{k}r_B)^{n+s+1}
\ls 2^{-k_0(\ez-s-1)}(r_B)^{s+1},
\end{eqnarray*}
which, together with $s\ge \lfr n(1/p-1)\rf$, gives
$$\frac{1}{|B|[\ro(|B|)]^2}\int_B|f(x)-P_{s,(r_B)^m}f(x)|^2\,dx\ls
2^{-k_0(\ez-s-1)}\frac{(r_B)^{2s+n+2}}{|B|^{2/p-1}}\ls 2^{-2k_0(\ez-s-1)}\ls \bz.$$
Hence, for any ball $B\subset (B(0,c))^{\com}$, we have that
$$\frac{1}{|B|[\ro(|B|)]^2}\int_B|f(x)-P_{s,(r_B)^m}f(x)|^2\,dx\ls \bz, $$
which, together with the fact that $\bz\rightarrow 0$ induces
$c\rightarrow \fz$, implies that $\gz_3(f)=0$. Thus, $f\in\vmo$. Therefore,
$\vmo=\lz_{n(1/p-1)}(\rn)$, which completes the proof of Corollary \ref{c3.1}.
\end{proof}

\begin{rem}\label{r3.2}\rm
If  $L=\sqrt\Delta$, similarly to the proof of Theorem \ref{t3.1}, we then have that
 for $p\in (n/(n+1),1]$ and $w(t)=t^p$ for all $t\in(0,\fz)$, the space
$\vmo$ coincides with the space $\lz_{n(1/p-1)}(\rn)$ with equivalent norms.
We omit the details.
\end{rem}

\section{Predual spaces of the spaces $B_{\oz,L}(\rn)$ \label{s4}}

\hskip\parindent In this section, we identify that the predual space of
$B_{\oz,L}(\rn)$ is $\vmoz$, where $B_{\oz,L}(\rn)$ is the Banach completion of the
space $\hw$; see Definition \ref{dp4.3} and Theorem \ref{t4.4} below.

\subsection{Tent spaces}\label{s4.1}
\hskip\parindent We begin with some notions and known facts on tent
spaces.

Recall that a function $a$ on $\rr^{n+1}_+$ is called a
$T_{\oz}({\rr}^{n+1}_+)$-atom if

 (i) there exists a ball $B\subset
\rn$ such that $\supp a\subset \widehat{B};$

 (ii) $\int_{\widehat{B}}|a(x,t)|^2\frac{\,dx\,dt}{t}\le
|B|^{-1}[\ro(|B|)]^{-2}.$

The following theorem was established in
\cite{hsv}.
\begin{thm}\label{t4.1}
Let $\oz$ satisfy the assumption (c). Then for any $f\in
T_\oz({\rr}^{n+1}_+)$, there exist a sequence $\{a_j\}_{j}$ of
$T_{\oz}({\rr}^{n+1}_+)$-atoms  and a sequence of
numbers, $\{\lz_j\}_j\subset \cc$, such that
\begin{equation}\label{4.1}
  f(x,t)=\sum_j\lz_ja_j(x,t)   \ \ \ \ \ \mbox{for \ a.\,e.}
  \ (x,\,t)\in \rr^{n+1}_+,
  \end{equation}
where the series converges pointwisely for almost every $(x,\,t)\in \rr^{n+1}_+.$
Moreover, there exists a positive constant $C$
 such that for all $f\in T_\oz({\rr}^{n+1}_+)$,
\begin{equation}\label{4.2}
\Lambda(\{\lz_ja_j\}_j)\equiv\inf\lf\{\lz>0:\,\sum_j|B_j|\oz\lf(\frac{|\lz_j|\|a_j
\|_{T_2^2({\rr}^{n+1}_+)}}{\lz|B_j| ^{1/2}}\r)\le1\r\}\le
C\|f\|_{T_\oz({\rr}^{n+1}_+)},
\end{equation}
where $\widehat{B_j}$ appears as the support of $ a_j$.
\end{thm}

\begin{defn}\label{dp4.1} Let $\oz$ satisfy the assumption (c).
The space $\wtw$ is defined to be the set of all $f=\sum_{j}\lz_j
a_j$, where the series converges in $(\twz)^\ast$, $\{a_j\}_j$ are $\twl$-atoms and
$\{\lz_j\}_j\in \ell^1.$ If $f\in\wtw$, then define
$$\|f\|_{\wtw}\equiv\inf\bigg\{\sum_j|\lz_j|\bigg\},$$
where the infimum is taken over all possible decompositions of $f$ as above.
\end{defn}

By \cite[Lemma 3.1]{hm1}, $\wtw$ is a Banach space.
Moreover, we have the following
lemma, which implies that $\twl$ is dense in $\wtw$.
Thus, in this sense, $\wtw$ is called the Banach completion
of the space $\twl$; see also \cite{ja, pe, wws}.

 \begin{lem}\label{l4.1}
 Let $\oz$ satisfy the assumption (c). Then there exists a positive constant
 $C$ such that for all $f\in\twl$, $f\in\wtw$ and
\begin{equation*}
 \|f\|_{\wtw}\le C\|f\|_{\twl}.
 \end{equation*}
\end{lem}
\begin{proof}\rm Let $f\in \twl$. By Theorem \ref{t4.1},  there exists
$T_{\oz}({\rr}^{n+1}_+)$-atoms $\{a_j\}_{j}$ and $\{\lz_j\}_j\subset \cc$ such
that \eqref{4.1} and \eqref{4.2} hold. Furthermore, by
the proof of Theorem \ref{t4.1}  in \cite{hsv}, we may choose $a_j$ such that
$\|a_j\|_{T^2_2(\rnz)}=|B_j|^{-1/2}\ro(|B_j|)^{-1}$ for each $j$, where
$\supp a_j\subset \widehat{B_j}$;  see also \cite{cms}.

For any $L\in \cn$, set $\sz_L\equiv\sum_{|j|\le L}|\lz_j|.$
Since $\oz$ is of upper type $1$, by this together with
$\ro(t)=t^{-1}/\oz^{-1}(t^{-1})$, we obtain
\begin{eqnarray*}&&\sum_j|B_j|\oz\lf(\frac{|\lz_j|\|a_j\|_{T_2^2({\rr}^{n+1}_+)}}
{\sigma_L|B_j|^{1/2}}\r)\gs \sum_{|j|\le L}|B_j|
\oz\lf(\frac{1}{|B_j|\ro(|B_j|)}\r)\frac{|\lz_j|}{\sigma_L} \gs
1,\end{eqnarray*}
which implies that
$$\sum_{|j|\le L} |\lz_j|\ls
\Lambda(\{\lz_ja_j\}_j)\ls\|f\|_{T_\oz({\rr}^{n+1}_+)}.$$
Letting $L\to\fz,$ we further obtain
$\sum_{j}|\lz_j|\ls\|f\|_{T_\oz({\rr}^{n+1}_+)}$.

Since $f\in \twl$ and $(\twl)^\ast=\twz$ (see \cite{jyz}),
we see that
$$f\in \twl\subset(\twl)^{\ast\ast}=(\twz)^\ast.$$
This fact can also be proved in the following direct way.
Indeed, for all $g\in\twz$, as a corollary of the monotone
convergence theorem and the H\"older inequality, we have
\begin{eqnarray*}
\int_{\rnz}\lf|f(x,t)g(x,t)\r|\,\frac{dx\,dt}{t}
&&\le\int_{\rnz}\sum_j|\lz_j|\lf|a_j(x,t)g(x,t)\r|\,\frac{dx\,dt}{t}\\
&&\le\sum_j|\lz_j|\int_{\widehat{B_j}}\lf|a_j(x,t)g(x,t)\r|\,\frac{dx\,dt}{t}\\
&&\le\sum_j|\lz_j|\|a_j\|_{T_2^2(\rnz)}\|g\chi_{\widehat{B_j}}\|_{T_2^2(\rnz)}\\
&&\le\|g\|_{\twz}\sum_j|\lz_j|\\
&&\ls \|f\|_{\twl}\|g\|_{\twz}<\fz,
\end{eqnarray*}
which together with the Lebesgue dominated convergence theorem
further implies that for all $g\in\twz$,
$$\lf|\int_{\rnz}f(x,t)g(x,t)\,\frac{dx\,dt}{t}\r|\ls \|f\|_{\twl}\|g\|_{\twz}.$$
That is, $f\in (\twz)^\ast$ and $\|f\|_{(\twz)^\ast}\ls\|f\|_{\twl}$.
Recall that for any $\ell\in (\twz)^\ast$, its $(\twz)^\ast$ norm is defined by
$$\lf\|\ell\r\|_{(\twz)^\ast}=\sup_{\|g\|_{\twz}\le 1}|\ell(g)|.$$
Observe also that $a_j\in (\twz)^\ast$ for all $j$. Now,
from these observations, the monotone convergence theorem
and the H\"older inequality, it follows that
\begin{eqnarray*}
&&\lf\|f-\sum_{|j|\le L} \lz_ja_j\r\|_{(\twz)^\ast}\\
&&\hs=\sup_{\|g\|_{\twz}\le 1}\lf|\int_{\rnz}\lf(f(x,t)-\sum_{|j|\le L} \lz_ja_j(x,t)
\r)g(x,t)\,\frac{dx\,dt}{t}\r|\\
&&\hs\le \sup_{\|g\|_{\twz}\le 1}
\int_{\twz}\sum_{|j|>L}|\lz_j||a_j(x,t)g(x,t)|\,\frac{dx\,dt}{t}\\
&&\hs=\sup_{\|g\|_{\twz}\le 1}\sum_{|j|>L}|\lz_j|
\int_{\widehat{B_j}}|a_j(x,t)g(x,t)|\,\frac{dx\,dt}{t}\\
&&\hs\le \sup_{\|g\|_{\twz}\le 1}\sum_{|j|>L}|\lz_j|
\|a_j\|_{T_2^2(\rnz)}\|g\chi_{\widehat{B_j}}\|_{T_2^2(\rnz)}
\le\sum_{|j|>L}|\lz_j|\to 0,
\end{eqnarray*}
as $L\to \fz$. Thus, the series in \eqref{4.1} converges in $(\twz)^\ast$,
which further implies that $f\in \wtw$ and
\begin{equation*}
\|f\|_{\wtw}\le\sum_{j}|\lz_j|\ls\|f\|_{T_\oz({\rr}^{n+1}_+)}.
\end{equation*}
This finishes the proof of Lemma \ref{l4.1}.
\end{proof}

\begin{lem}\label{l4.2}Let $\oz$ satisfy the assumption (c).
Then $(\wtw)^\ast=\twz$ via the pairing
$$\la f,g\ra=\int_{\rnz}f(x,t)g(x,t)\frac{\,dx\,dt}{t}$$
for all $f\in\wtw$ and $g\in \twz$.
\end{lem}
\begin{proof}\rm  The facts that $(\twl)^\ast=\twz$ (see \cite{jyz}) and
$\twl\subset\wtw$ imply that $(\wtw)^\ast\subset\twz$.

Conversely, let $g\in\twz$. Then for any $f\in\wtw$,
 choose a sequence $T_\oz$-atoms $\{a_j\}_j$ and $\{\lz_j\}_j\in \ell^1$ such
 that $f=\sum_j\lz_ja_j$, where the series converges in $(\twz)^\ast$ and
 $\sum_j|\lz_j|\ls \|f\|_{\wtw}$. Thus, by the H\"older inequality, we obtain
\begin{eqnarray*}
|\la f,g\ra|&&\le\sum_j|\lz_j|\int_{\rnz}|a_j(x,t)g(x,t)|\,
\frac{dx\,dt}{t}\\&&\le \|g\|_{\twz}\sum_{j}|\lz_j|\ls
\|g\|_{\twz}\|f\|_{\wtw},\end{eqnarray*} which implies that $g\in
(\wtw)^\ast$, and thus, completes the proof of Lemma \ref{l4.2}.
\end{proof}

Recall that $\twc$ denotes the set of all functions in $\twz$ with compact support.

\begin{lem}\label{l4.3} Let $\oz$ satisfy the assumption (c). If $f\in \wtw$, then
\begin{eqnarray}\label{4.3}
  \|f\|_{\wtw}=\sup_{g\in \twc,\,\|g\|_{\twz}\le 1}
  \bigg|\int_{\rnz}f(x,t)g(x,t)\frac{\,dx\,dt}{t}\bigg|.
\end{eqnarray}
\end{lem}

\begin{proof}\rm Let $f\in\wtw$. By Lemma \ref{l4.2}, we obtain
\begin{eqnarray*}
  \|f\|_{\wtw}=\sup_{\|g\|_{\twz}\le 1}
  \bigg|\int_{\rnz}f(x,t)g(x,t)\frac{\,dx\,dt}{t}\bigg|.
\end{eqnarray*}
For any $\bz>0$, there exists $g\in \twz$ with
 $\|g\|_{\twz}\le 1$ such that
\begin{eqnarray*}
\bigg|\int_{\rnz}f(x,t)g(x,t)\frac{\,dx\,dt}{t}\bigg|\ge
\|f\|_{\wtw}-\frac{\bz}{2}.
\end{eqnarray*}
Notice here $fg\in L^1(\rnz)$. Let $O_k\equiv\{(x,t):\,
|x|<k, 1/k<t<k\}$. Then there exists $k\in\cn$ such that
\begin{eqnarray*}
\bigg|\int_{\rnz}f(x,t)g(x,t)\chi_{O_k}(x,t)\frac{\,dx\,dt}{t}\bigg|\ge
\|f\|_{\wtw}-\bz.
\end{eqnarray*}
Obviously, $g\chi_{O_k}\in \twc$. Thus, \eqref{4.3} holds,
which completes the proof of Lemma \ref{l4.3}.
\end{proof}

The following lemma is a slight modification of \cite[Lemma 4.2]{cw}; see also
\cite{pe,wws}. We omit the details here.

\begin{lem}\label{l4.4}
Let $\oz$ satisfy the assumption (c). Suppose that
$\{f_k\}_{k=1}^\fz$ is a bounded family of functions in $\wtw$. Then
there exist $f\in \wtw$ and a subsequence $\{f_{k_j}\}_j$ such that
for all $g\in \twc$,
$$\lim_{j\rightarrow\fz}\int_{\rnz}f_{k_j}(x,t)g(x,t)\frac{\,dx\,dt}{t}=
\int_{\rnz}f(x,t)g(x,t)\frac{\,dx\,dt}{t}.$$
\end{lem}

\begin{thm}\label{t4.2}
Let $\oz$ satisfy the assumption (c).
Then $(\twv)^\ast$, the dual space of the space $\twv$, coincides
with $\wtw$ in the following sense:

(i) For any  $g\in \wtw$, define the linear functional $\ell$ by setting,
for all $f\in \twv$,
\begin{equation}\label{4.4}
\ell(f)\equiv \int_{\rnz}f(x,t)g(x,t)\frac{\,dx\,dt}{t}.
\end{equation}
Then there exists a positive constant $C$, independent of $g$,
such that $$\|\ell\|_{(\twv)^\ast}\le
C\|g\|_{\wtw}.$$

(ii) Conversely, for any $\ell\in (\twv)^\ast$, there exists $g\in
\wtw$ such that \eqref{4.4} holds for all $f\in \twv$ and
$\|g\|_{\wtw}\le C\|\ell\|_{(\twv)^\ast}$, where $C$ is independent
of $\ell.$
\end{thm}
\begin{proof}\rm By Lemma \ref{l4.2}, we obtain $(\wtw)^\ast=\twz\supset \twv$,
which further implies that $\wtw\subset(\wtw)^{\ast\ast}\subset
(\twv)^\ast$.

Conversely, let $\ell \in (\twv)^\ast$. Notice that
for any $f\in T_{2,c}^2(\rnz)$, we may assume that $\supp f\subset K$, where $K$ is a
compact set in $\rnz$. Then we have
$$\|f\|_{\twv}=\|f\|_{\twz}\le C(K)\|f\|_{T_2^2(\rnz)}.$$
 Thus, $\ell$ induces a bounded linear
functional on $T_{2,c}^2(K)$. Let $O_k$ be as in the proof of Lemma \ref{l4.3}.
 By the Riesz theorem, there exists a unique
$g_k\in L^2(O_k)$ such that for all $f\in L^2(O_k)$,
$$\ell (f)=\int_{{\rr}^{n+1}_+} f(x,\,t)g_k(x,\,t)\,\frac{dx\,dt}{t}.$$
Obviously, $g_{k+1}\chi_{O_k}=g_k$ for all $k\in \cn.$ Let
$g\equiv g_1\chi_{O_1}+\sum_{k=2}^\fz g_k\chi_{O_k\backslash
 O_{k-1}}.$ Then $g\in L^2_{\loc}({\rr}^{n+1}_+)$
and for any $f\in T_{2,c}^2({\rr}^{n+1}_+)$, we have
$$\ell(f)= \int_{{\rr}^{n+1}_+}f(y,t)g(y,t)\frac{\,dy\,dt}{t}.$$

Set $\wz g_k\equiv g\chi_{O_k}$. Then for each $k\in\cn$, obviously, we have
$\wz g_k\in T_{2,c}^2(\rnz)\subset\twl\subset\wtw$. Then by Lemma \ref{l4.3}, we
have
\begin{eqnarray*}
  \|\wz g_k\|_{\wtw}&&=\sup_{f\in \twc,\,\|f\|_{\twz}\le 1}
  \bigg|\int_{\rnz}f(x,t)g(x,t)\chi_{O_k}(x,t)\frac{\,dx\,dt}{t}\bigg|\\
  &&=\sup_{f\in \twc,\,\|f\|_{\twz}\le 1}|\ell(f\chi_{O_k})|\\
  &&\le \sup_{f\in \twc,\,\|f\|_{\twz}\le 1}\|\ell\|_{(\twv)^\ast}\|f\|_{\twz}
  \le \|\ell\|_{(\twv)^\ast}.
\end{eqnarray*}
Thus, by Lemma \ref{l4.4}, there exists $\wz g\in\wtw$ and
$\{\wz g_{k_j}\}_j\subset \{\wz g_{k}\}_k$ such that for all $f\in\twc$,
$$\lim_{j\rightarrow\fz}\int_{\rnz}f(x,t)\wz g_{k_j}(x,t)\frac{\,dx\,dt}{t}=
\int_{\rnz}f(x,t)\wz g(x,t)\frac{\,dx\,dt}{t}.$$
On the other hand, notice that for sufficient large $k_j$, we have
$$\ell(f)=\int_{\rnz}f(x,t)g(x,t)\frac{\,dx\,dt}{t}=
\int_{\rnz}f(x,t)\wz g_{k_j}(x,t)\frac{\,dx\,dt}{t}
=\int_{\rnz}f(x,t)\wz g(x,t)\frac{\,dx\,dt}{t},$$
which implies that $g=\wz g$ almost everywhere, hence $g\in\wtw$.
By a density argument, we know that \eqref{4.4} also holds for $g$
and all $f\in\twv$.
\end{proof}

\subsection{Dual spaces of $\vmo$ }\label{s4.2}
\hskip\parindent In this subsection, we identify the dual space of
$\mathrm{VMO}_{\ro,\,L^\ast}(\rn)$ .

Now, let $L^{2}_c({\rr}^{n+1}_+)$ denote the set of
$L^{2}({\rr}^{n+1}_+)$ functions with compact support. Then for all
$f\in L^{2}_c({\rr}^{n+1}_+)$, define $\pi_{L,s,s_1}$ by setting,
for all $x\in\rn$,
\begin{equation}\label{4.5}
\pi_{L,s,s_1}(f)(x)\equiv
C_{m,s,s_1}\int^\fz_0Q_{s,\,t^m}(I-P_{s_1,\,t^m})(f(\cdot,t))(x)\frac{\,dt}{t},
\end{equation}
where $C_{m,s,s_1}$ is the same as in \eqref{3.2}.
From  \eqref{2.4}, it is easy to deduce that $\pi_{L,s,s_1}$ is well
defined on $L^{2}_c({\rr}^{n+1}_+)$; indeed, by the H\"older inequality,
for any $f\in L^{2}_c({\rr}^{n+1}_+)$, we have $\pi_{L,s,s_1}(f)\in L^2(\rn).$

The following notion of molecules is introduced in \cite{jyz};
see also \cite{adm2,ya2}.

\begin{defn}\label{dp4.2}
 Let $\oz$ satisfy the assumption (c) and $s\ge s_1\ge\lfr\frac
nm(\frac{1}{\wz p_0(\oz)}-1)\rf.$ A function $\az$  on $\rn$ is called
an $(\oz,s,s_1)$-molecule if there exists a
$T_\oz({\rr}^{n+1}_+)$-atom $a$ such that for all $x\in\rn$,
\begin{equation*}
\az(x)\equiv
\pi_{L,s,s_1}(a)(x)=C_{m,s, s_1}\int^\fz_0Q_{s,\,t^m}(I-P_{s_1,t^m})
(a(\cdot,\,t))(x)\frac{\,dt}{t}.
\end{equation*}
\end{defn}

The following result is essentially Theorem 4.3 and Theorem 4.6 in \cite{jyz}.

\begin{thm}\label{t4.3}
Let $\oz$ satisfy the assumption (c) and $s\ge s_1\ge\lfr\frac
nm(\frac{1}{\wz p_0(\oz)}-1)\rf.$ Then for all $f\in
H_{\oz,\,L}(\rn)$,
 there exist a family $\{\az_k\}_k$ of
$(\oz,s,s_1)$-molecules  and $\{\lz_k\}_k\subset \cc$ such that
$f=\sum_k\lz_k\az_k,$
where the series converges in $H_{\oz,\,L}(\rn).$
Furthermore, if define
\begin{equation*}
\widetilde{{\Lambda}}(\{\lz_k\az_k\}_k) \equiv\inf\bigg\{\Lambda(\{\mu_k
a_k\}_k)\bigg\},
\end{equation*}
where the infimum is taken over all possible
$T_{\oz}({\rr}^{n+1}_+)$-atoms $\{a_k\}_k$ and
$\{\mu_k\}_k\subset\cc$ such that
$\sum_k\mu_k\pi_{L,s,s_1}(a_k)=\sum_k\lz_k\az_k$ in $H_{\oz,L}(\rn)$, then
there exists a positive constant $C$ independent of $f$
 such that $\widetilde{{\Lambda}}(\{\lz_k\az_k\}_k)\le
C\|f\|_{H_{\oz,L}(\rn)}.$

Conversely, suppose that  $\{\az_k\}_{k}$ is a family of
$(\oz,s,s_1)$-molecules and $\{\lz_k\}_{k}\subset\cc$ such that
${{\Lambda}}(\{\lz_k a_k\}_k)<\fz$ with $\az_k=\pi_{L,s,s_1}(a_k)$ for all
$k.$ Then $\sum_k\lz_k\az_k\in H_{\oz,L}(\rn)$ with
\begin{equation*}
\lf\|\sum_k\lz_k\az_k\r\|_{H_{\oz,L}(\rn)} \le
C\widetilde{{\Lambda}}(\{\lz_k\az_k\}_k),
\end{equation*} where $C$ is
a positive constant independent of $\{\lz_k\}_k$ and $\{\az_k\}_{k}.$
\end{thm}

Let us now introduce the Banach completion of the space $\hw$.

\begin{defn}\label{dp4.3}
 Let the assumptions (a), (b) and (c) hold and $s\ge s_1\ge
\lfr\fss\rf.$
The space $\hv$ is defined to be the set of all $f=\sum_j\lz_j\az_j$,
where the series converges in $(\bmoz)^\ast$, $\{\lz_j\}_j\in
\ell^1$ and $\{\az_j\}_j$ are $(\oz,s,s_1)$-molecules. If $f\in\hv$, define
\begin{equation*}
  \|f\|_{\hv}=\inf\bigg\{\sum_j|\lz_j|\bigg\},
\end{equation*}
where the infimum is taken over all the possible decompositions of $f$ as above.
\end{defn}

By \cite[Lemma 3.1]{hm1} again, we see that $\hv$ is a Banach space.
Similarly to the proof of Lemma \ref{l4.1}, we have the
following embedding of $\hw$ into $\hv$. We omit the details here.

\begin{lem}\label{l4.5} Let the assumptions (a), (b) and (c) hold and $s\ge s_1\ge
\lfr\fss\rf$. Then there exists a positive constant $C$ such that for all $f\in
\hw$, $f\in\hv$ and
\begin{equation*}
 \|f\|_{\hv}\le  C\|f\|_{\hw}.
 \end{equation*}
\end{lem}

As a corollary of Lemma \ref{l4.5}, we have the following basic properties of
the space $\hv$.

\begin{prop}\label{p4.1} Let the assumptions (a), (b) and (c) hold and $s\ge s_1\ge
\lfr\fss\rf$. Then

(i) The space $\hw$ is dense in $\hv$.

(ii) For any $\wz s\ge \wz s_1\ge \lfr\fss\rf$, the spaces $\hv$ and
$B_{\oz,L}^{\wz s, \wz s_1}(\rn)$ coincide with equivalent norms.
\end{prop}

\begin{proof}\rm
By Lemma \ref{4.5}, we have $\hw\subset B_{\oz,L}(\rn)$. On the other hand,
notice that the set of all finite linear combinations of $(\oz,s,s_1)$-molecules
is included in $\hw$ and is also dense in $\hv$. Thus, (i) holds.

To prove (ii), for any $f\in\hv$, by Definition \ref{dp4.3},
there exist $(\oz,s,s_1)-$molecules
$\{\az_j\}_j$ and $\{\lz_j\}_j\in \ell^1$ such that $f=\sum_j\lz_j\az_j$,
where the series converges in $(\bmor)^\ast$, and
$\sum_j|\lz_j|\ls\|f\|_{\hv}.$
By Theorem \ref{t4.3}, we have $\az_j\in\hw$ and
$\|\az_j\|_{\hw}\ls 1$, which together with
Lemma \ref{l4.5} implies that $\az_j\in B_{\oz,L}^{\wz s, \wz s_1}(\rn)$
and
$$\|\az_j\|_{B_{\oz,L}^{\wz s, \wz s_1}(\rn)}\ls \|\az_j\|_{\hw}\ls 1.$$
Then by Definition \ref{dp4.3} again,
there exists $(\oz,\wz s,\wz s_1)-$molecules $\{\az_{j,k}\}_k$ and
$\{\lz_{j,k}\}_k\in \ell^1$ such that
$\az_j=\sum_k\lz_{j,k}\az_{j,k}$, where the series converges
in $(\bmor)^\ast$ and $\sum_k|\lz_{j,k}|\ls 1$.
Thus, finally, we obtain that
$$f=\sum_j\lz_j\az_j=\sum_{j}\lz_j\lf\{\sum_{k}\lz_{j,k}\az_{j,k}\r\},$$
which holds in $(\bmor)^\ast$ and
$$\|f\|_{B_{\oz,L}^{\wz s, \wz s_1}(\rn)}\le\sum_{j}|\lz_j|
\lf\{\sum_k|\lz_{j,k}|\r\}\ls \sum_{j}|\lz_j|\ls \|f\|_{\hv}.$$
Thus, $\hv\subset B_{\oz,L}^{\wz s, \wz s_1}(\rn)$. By symmetry,
we also have $B_{\oz,L}^{\wz s, \wz s_1}(\rn)\subset\hv$,
which completes the proof of Proposition \ref{p4.1}.
\end{proof}

Proposition \ref{p4.1} (ii) implies that the space
$\hv$ with $s,\,s_1$ as in Definition \ref{dp4.3}
is independent of the choices of $s$ and $s_1$.
Based on this, from now on, for any
$s\ge s_1\ge \lfr\fss\rf$, we denote
the space $\hv$ simply by $B_{\oz,L}(\rn)$.
Also, by Proposition \ref{p4.1} (i), we call
the space $B_{\oz,L}(\rn)$ the Banach completion of
the Orlicz-Hardy space $\hw$; see also Definition
\ref{dp4.1} and its following remarks.

The proof of the following lemma is similar to that of Lemma \ref{l4.2}.
We omit the details.

\begin{lem}\label{l4.6} Let the assumptions (a), (b) and (c) hold.
Then $(B_{\oz,L}(\rn))^\ast$ coincides with $\bmoz$ in the following sense:

(i) For any  $g\in \mathrm{BMO}_{\ro,L^\ast}(\rn)$, the
linear functional $\ell$, which is initially defined on
$B_{\oz,L}(\rn)\cap L^2(\rn)$ by setting, for all $f\in B_{\oz,L}(\rn)\cap L^2(\rn)$,
\begin{equation}\label{4.6}
\ell(f)\equiv \int_{\rn}f(x)g(x)\,dx,
\end{equation}
has a unique extension to $B_{\oz,L}(\rn)$ with
$\|\ell\|_{(B_{\oz,L}(\rn))^\ast}\le
C\|g\|_{\mathrm{BMO}_{\ro,L^\ast}(\rn)},$ where $C$ is a
positive constant independent of $g.$

(ii) Conversely, for any $\ell\in (B_{\oz,L}(\rn))^\ast$, there
exists $g\in \mathrm{BMO}_{\ro,\,L^\ast}(\rn)$ such that
\eqref{4.6} holds for all $f\in B_{\oz,L}(\rn)\cap L^2(\rn)$ and
$\|g\|_{\mathrm{BMO}_{\ro,\,L^\ast}(\rn)}\le
C\|\ell\|_{(B_{\oz,L}(\rn))^\ast}$, where $C$ is independent of
$\ell.$
\end{lem}

\begin{lem}\label{l4.7}
Let $\oz$ satisfy the assumption (c). Then $T_{2,c}^2(\rnz)$ is dense in $\wtw$.
\end{lem}
\begin{proof}\rm
Since $\twl$ is dense in $\wtw$, to prove the lemma, it suffices to prove that
$T_{2,c}^2(\rnz)$ is dense in $\twl$ in the norm $\|\cdot\|_{\wtw}$.

For any $g\in \twl$, let $g_k\equiv g\chi_{O_k}$, where $O_k$ is the same as in the
proof of Lemma \ref{l4.3}. Then $g_k\in T_{2,c}^2(\rnz)$. By the dominated convergence
theorem and the continuity of $\oz$, we have that for any $\lz>0$,
$$\lim_{k\rightarrow\fz}\int_{\rn}\oz\lf(\frac{\ca(g-g_k)(x)}{\lz}\r)\,dx
=\int_{\rn}\lim_{k\rightarrow\fz}\oz\lf(\frac{\ca(g-g_k)(x)}{\lz}\r)\,dx=0,$$
which implies that $\lim_{k\rightarrow\fz}\|g-g_k\|_{\twl}=0$. Then, by Lemma
\ref{l4.1}, we have
$$\|g-g_k\|_{\wtw}\ls \|g-g_k\|_{\twl}\rightarrow 0,$$
as $k\rightarrow\fz$, which completes the proof of Lemma \ref{l4.7}.
\end{proof}

\begin{lem}\label{l4.8}
Let the assumptions (a), (b) and (c) hold and $s\ge s_1\ge
\lfr\frac nm(\frac{1}{\wz p_0(\oz)}-1)\rf.$ Then the  operator
$\pi_{L,s,s_1}$, initially defined on $L^2_c(\rnz)$, extends to a
bounded linear operator:

(i) from $T_{2}^p(\rnz)$ to $L^p(\rn)$, if $p\in(1,\fz)$; or

(ii) from $\twl$ to $\hw$; or

(iii) from $\wtw$ to $B_{\oz,L}(\rn)$.
\end{lem}
\begin{proof} (i) was established in \cite{ya2} and
(ii) was established in \cite{jyz}. Let us now prove (iii).

 Let $f\in T_{2,c}^2(\rnz).$ Since $T_{2,c}^2(\rnz)\subset\twl$, by (ii), we have
$\pi_{L,s,s_1}(f)\in \hw\subset\hv$. Moreover, since $\twl\subset\wtw$,
by Definition \ref{dp4.1}, there exist
$T_\oz(\rnz)$-atoms $\{a_j\}_j$ and $\{\lz_j\}_j\subset\cc$ such that
 $f=\sum_j \lz_ja_j$, where the series converges in $(\twz)^\ast$, and
$\sum_j|\lz_j|\ls\|f\|_{\wtw}.$
Thus, for any $g\in\bmoz$, by Theorem \ref{t3.1}, we obtain
\begin{eqnarray*}
\la\pi_{L,s,s_1}(f),g\ra&&=\int_{\rnz}f(x,t)Q^\ast_{s,t^m}
(I-P^\ast_{s_1,t^m})(g)(x)\,\frac{dx\,dt}{t}\\
&&=\sum_j\lz_j\int_{\rnz}a_j(x,t)Q^\ast_{s,t^m}
(I-P^\ast_{s_1,t^m})(g)(x)\,\frac{dx\,dt}{t}\\
&&=\sum_j\lz_j\la\pi_{L,s,s_1}(a_j),g\ra,
\end{eqnarray*}
which implies that $\pi_{L,s,s_1}(f)=\sum_j\lz_j\pi_{L,s,s_1}(a_j)$
holds in $(\bmoz)^\ast$ and hence,
 $$\|\pi_{L,s,s_1}(f)\|_{\hv}\le \sum_j|\lz_j|\ls \|f\|_{\wtw}.$$
Since by Lemma \ref{l4.7}, $T_{2,c}^2(\rnz)$ is dense in $\wtw$,
we then obtain (iii) by a density argument, and thus  complete the proof of Lemma
\ref{l4.8}.
\end{proof}

We also define the operator $\pi_{L}$ by setting, for all
$f\in L^{2}_c({\rr}^{n+1}_+)$ and $x\in\rn$,
\begin{equation*}
\pi_{L}(f)(x)\equiv
C_{m}\int^\fz_0Q_{t^m}(f(\cdot,t))(x)\frac{\,dt}{t},
\end{equation*}
where $C_{m}=4m$. Similarly to \eqref{4.5}, from \eqref{2.4},
it is easy to deduce that the operator $\pi_{L}$ is well defined.
\begin{lem}\label{l4.9}
Let the assumptions (a), (b) and (c) hold. Then the operator
$\pi_{L}$, initially defined on $L^2_c(\rnz)$, extends to a
bounded linear operator:

(i) from $T_{2}^p(\rnz)$ to $L^p(\rn)$, if $p\in(1,\fz)$; or

(ii) from $\twv$ to $\vmo$.
\end{lem}

\begin{proof}\rm (i) was established in \cite[Lemma 4.3]{dy2}.
To prove (ii), let $ s_1\ge \lfr\frac nm(\frac{1}{\wz p_0(\oz)}-1)\rf$, $s\ge 2s_1$,
$\ez\in (n\bz_1(\ro), \tz(L))$ and $\wz\ez=(\ez-n\bz_1(\ro))/2$.
Suppose that  $f\in\twv$. To see $\pi_{L}(f)\in\vmo$,
 by Theorem \ref{t3.2}, it suffices to show that
$Q_{s,t^m}(I-P_{2s_1,t^m})\pi_{L}(f)\in \twv$. We first show
that for any ball $B$,
\begin{equation}\label{4.7}
\frac{1}{\ro(|B|)|B|^{1/2}}\bigg(\int_{\widehat
B}|Q_{s,\,t^m}(I-P_{2s_1,t^m})\pi_{L}(f)(x)|^2\frac{\,dx\,dt}{t}\bigg)^{1/2}\ls
\sum_{k=1}^\fz 2^{-k\wz\ez}\sz_k(f,B),
\end{equation}
where
$$\sz_k(f,B)\equiv\frac{1}{|2^kB|^{1/2}\ro(|2^kB|)}
\bigg(\int_{\widehat{2^{k+1}B}}|f(x,t)|^2\frac{\,dx\,dt}{t}\bigg)^{1/2}.$$
Once \eqref{4.7} is proved, then by an argument similar to that used in the proof
of Theorem \ref{t3.2} (see the proof of \eqref{3.4}), we obtain that
$Q_{s,\,t^m}(I-P_{2s_1,t^m})\pi_{L}(f)\in\twv$ and hence,
$\pi_{L}(f)\in\vmo$.

To prove \eqref{4.7}, let $f_1\equiv f\chi_{\widehat{4B}}$ and
$f_2\equiv f\chi_{(\widehat{4B})^\com}$. Then
\begin{eqnarray*}
  &&\bigg(\int_{\widehat
B}|Q_{s,\,t^m}(I-P_{2s_1,t^m})\pi_{L}(f)(x)|^2\frac{\,dx\,dt}{t}\bigg)^{1/2}
\\&&\quad\le \sum_{i=1}^2\bigg(\int_{\widehat
B}|Q_{s,\,t^m}(I-P_{2s_1,t^m})\pi_{L}(f_i)(x)|^2\frac{\,dx\,dt}{t}\bigg)^{1/2}
\equiv\sum_{i=1}^2 \mathrm{I}_i.
\end{eqnarray*}

For the term $\mathrm{I}_1$, by \eqref{2.4} and (i) with $p=2$, we obtain
 $$\mathrm{I}_1\ls \|\pi_{L}(f_1)\|_{L^2(\rn)}\ls
 \|f_1\|_{T^2_2(\rnz)}\sim |2B|^{1/2}\ro(|2B|)\sz_2(f,B).$$

To estimate $\mathrm{I}_2$, write
\begin{eqnarray*}
Q_{s,\,t^m}(I-P_{2s_1,t^m})\pi_{L}(f_2)(x)&&= \int_0^\fz
Q_{s,\,t^m}(I-P_{2s_1,t^m})Q_{\nu^m}(f_2)(x)\frac{\,d\nu}{\nu}\\
&&\equiv  \int_0^\fz \Psi_{t,\nu}(L)(f_2)(x)\frac{\,d\nu}{\nu}.
\end{eqnarray*}
It follows from \eqref{2.8} that $\psi_{t,\nu}$, the kernel of
$\Psi_{t,\nu}(L)$, satisfies that for all $x,\,y\in\rn$,
$$|\psi_{t,\nu}(x,y)|\ls
\frac{t^{m(s+1)}\nu^m}{(t+\nu)^{m(s+2)}}\frac{(t+\nu)^\ez}{(t+\nu+|x-y|)^{n+\ez}}
\ls \frac{t^{\ez-\wz\ez}\nu^{\wz\ez}}{(t+\nu+|x-y|)^{n+\ez}}.$$
Moreover, for $k\ge 2$ and $(x,\nu)\in
\widehat{2^{k+1}B}\backslash \widehat{2^{k}B}$, we have
$t+\nu+|x-y|\sim 2^kr_B$. By this and the H\"older inequality,
we deduce that
\begin{eqnarray*}
\mathrm{I}_2&\ls&
\sum_{k=2}^\fz\lf(\int_{\widehat
B}\bigg[\int_{\widehat{2^{k+1}B}\backslash
\widehat{2^{k}B}}|\psi_{t,\nu}(x,y)||f(y,\nu)|\frac{\,dyd\nu}{\nu}\bigg]^2\frac{\,dx
dt}{t}\r)^{1/2}\\
&\ls& \sum_{k=2}^\fz\lf(\int_{\widehat
B}\bigg[\int_{\widehat{2^{k+1}B}\backslash
\widehat{2^{k}B}}\frac{t^{\ez-\wz\ez}\nu^{\wz\ez}}
{(t+\nu+|x-y|)^{n+\ez}}|f(y,\nu)|\frac{\,dyd\nu}{\nu}\bigg]^2\frac{\,dx
dt}{t}\r)^{1/2}\\
&\ls&\sum_{k=2}^\fz\lf(\int_{\widehat
B}\bigg[\int_{\widehat{2^{k+1}B}\backslash
\widehat{2^{k}B}}\nu^{2\wz\ez}
\frac{\,dyd\nu}{\nu}\bigg]t^{2\ez-2\wz\ez}\frac{\,dx
dt}{t}\r)^{1/2}(2^kr_B)^{-n-\ez}|2^kB|^{1/2}\ro(|2^kB|)\sz_k(f,B)\\
&\ls&\sum_{k=2}^\fz\lf[(2^kr_B)^{n+2\wz\ez}
r_B^{n+2\ez-2\wz\ez}\r]^{1/2}(2^kr_B)^{-n/2-\ez}2^{kn\bz_1(\ro)}\ro(|B|)\sz_k(f,B)\\
&\ls&\sum_{k=2}^\fz
2^{-k(\ez-n\bz_1(\ro))/2}\ro(|B|)|B|^{1/2}\sz_k(f,B).
\end{eqnarray*}
Thus, \eqref{4.7} is proved, which completes the proof of Lemma
\ref{l4.9}.
\end{proof}

\begin{lem}\label{l4.10} Let the assumptions (a), (b) and (c) hold and
$\ro$ be as in \eqref{2.10}. Then $\vmo\cap L^2(\rn)$ is dense
in $\vmo$.
\end{lem}

\begin{proof}\rm
Let $s_1\ge\lfr\frac nm(\frac{1}{\wz
p_0(\oz)}-1)\rf$ and $s\ge 2s_1$.
Notice that for any $f\in L^2(\rn)$, we have
$$f=\pi_{L,s,2s_1}(f)=C_{m,s,2s_1}\int_0^\fz
Q_{s,t^m}(I-P_{2s_1,t^m})Q_{t^m}f\frac{\,dt}{t}$$ in
$L^2(\rn)$, where $C_{m,s,2s_1}$ is as in \eqref{3.2}.

Let $g\in \vmo$. For any $f\in\hz\cap L^2(\rn)$, by Lemma \ref{l3.1}, we have
\begin{eqnarray*}\int_{\rn}f(x)g(x)\,dx&&=C_{m,s,2s_1}\int_{\rnz}
Q^\ast_{t^m}f(x)Q_{s,t^m}
(I-P_{2s_1,t^m})g(x)\,\frac{dx\,dt}{t}.
\end{eqnarray*}
Since $g\in \vmo$, by Theorem \ref{t3.2}, we know that $Q_{s,t^m}
(I-P_{2s_1,t^m})g\in\twv$. Let $O_k$ be as in the proof of Lemma \ref{4.3},
$h\equiv Q_{s,t^m}(I-P_{2s_1,t^m})g$ and
$h_k\equiv\chi_{O_k}h$ for each $k\in\cn$.
Then $h_k\in \twc=T_{2,c}^2(\rnz)$ and
$\|h_k-h\|_{\twz}\rightarrow0$, as $k\rightarrow\fz$.
Thus, by Lemma \ref{l4.9}, we have that $\pi_L(h_k)\in
 \vmo\cap L^2(\rn)$ and
 \begin{equation}\label{4.8}
\|\pi_L(h-h_k)\|_{\bmor}\ls \|(h-h_k)\|_{\twz}\rightarrow 0,
 \end{equation}
as $k\rightarrow\fz$. Then by the dominated
convergence theorem, we further have
\begin{eqnarray*}
\int_{\rn}f(x)g(x)\,dx&&=C_{m,s,2s_1}\int_{\rnz}
Q^\ast_{t^m}f(x)h(x,t)\,\frac{\,dx\,dt}{t}\\
&&=C_{m,s,2s_1}\lim_{k\rightarrow\fz}\int_{\rnz}
Q^\ast_{t^m}f(x)h_k(x,t)\,\frac{dx\,dt}{t}\\
&&=C_{m,s,2s_1}\lim_{k\rightarrow\fz}
\int_{\rn}f(x)\int_0^\fz Q_{t^m}(h_k)(x)\,\frac{dx\,dt}{t}\\
&&=\frac{C_{m,s,2s_1}}{C_m}\lim_{k\rightarrow\fz}
\la f,\pi_L(h_k)\ra=\frac{C_{m,s,2s_1}}{C_m}\la f,\pi_L(h)\ra.
\end{eqnarray*}
Since $\hz\cap L^2(\rn)$ is dense in the space $\hz$, we then obtain that
\begin{equation*}
g=\frac{C_{m,s,2s_1}}{C_m} \pi_{L}(h)=\frac{C_{m,s,2s_1}}{C_m}
\pi_{L}(Q_{s,t^m}(I-P_{2s_1,t^m})g)
\end{equation*}
in $\vmo$. Let $g_k\equiv\frac{C_{m,s,2s_1}}{C_m}\pi_L(h_k)$. Then
by Lemma \ref{l4.9} again, we have that $g_k\in \vmo\cap L^2(\rn)$ and
by \eqref{4.8}, $\|g-g_k\|_{\bmor}\rightarrow 0$, as $k\rightarrow \fz$,
which completes the proof of Lemma \ref{l4.10}.
\end{proof}

The following is the main result of this paper.
Recall that by Definition \ref{dp3.1},
$$\mathrm{VMO}_{\ro,\,L^\ast}(\rn)\subset
\mathrm{BMO}_{\ro,\,L^\ast}(\rn).$$
The symbol $\la\cdot,\cdot\ra$ in the following theorem
means the duality between $\mathrm{BMO}_{\ro,\,L^\ast}(\rn)$
and the space $B_{\oz, L}(\rn)$ in the sense of Lemma \ref{l4.6}.

\begin{thm}\label{t4.4}
Let the assumptions (a), (b) and (c) hold.
Then $(\mathrm{VMO}_{\ro,\,L^\ast}(\rn))^\ast=
B_{\oz, L}(\rn)$ in the following sense:

(i) For any $g\in B_{\oz, L}(\rn)$, define the linear functional $\ell$ by setting,
for all $f\in \mathrm{VMO}_{\ro,\,L^\ast}(\rn)$,
\begin{equation}\label{4.9}
\ell(f)\equiv \la f, g\ra.
\end{equation}
Then there exists a positive constant $C$,
independent of $g$, such that $\|\ell\|_{(\mathrm{VMO}_{\ro,\,L^\ast}(\rn))^\ast}
\le C\|g\|_{B_{\oz, L}(\rn)}$.

(ii) Conversely, for any $\ell\in(\mathrm{VMO}_{\ro,\,L^\ast}(\rn))^\ast$,
there exists $g\in B_{\oz, L}(\rn)$ such that \eqref{4.9} holds and
there exists  a positive constant $C$, independent of $\ell$,
such that  $\|g\|_{B_{\oz, L}(\rn)}\le
C\|\ell\|_{(\mathrm{VMO}_{\ro,L^\ast}(\rn))^\ast}$.
\end{thm}

\begin{proof}\rm
Since $\mathrm{VMO}_{\ro,\,L^\ast}(\rn)\subset\bmoz$,
from Lemma \ref{l4.6}, it is easy to see that
 $B_{\oz, L}(\rn)\subset(\mathrm{VMO}_{\ro,\,L^\ast}(\rn))^\ast$.

Conversely, let $s_1\ge\lfr\fss\rf$, $s\ge 2s_1$ and
$\ell\in (\mathrm{VMO}_{\ro,\,L^\ast}(\rn))^\ast$.
For any $f\in \mathrm{VMO}_{\ro,\,L^\ast}(\rn)\cap L^2(\rn)$, by
Theorem \ref{t3.2}, we have $Q^\ast_{s,t^m}(I-P^\ast_{2s_1,t^m})f\in\twv$. On the
other hand, from Lemma \ref{l4.8},
it follows that $\ell\circ\pi_{L^\ast}\in (\twv)^\ast$. By this
together with Theorem \ref{t4.2}, we obtain that there
exists $g\in \wtw$ such that
\begin{eqnarray*}\ell(f)&&=\frac{C_{m,s,2s_1}}{C_{m}}
\ell\circ\pi_{L^\ast}(Q^\ast_{s,t^m}(I-P^\ast_{2s_1,t^m})f)\\
&&=\frac{C_{m,s,2s_1}}{C_{m}}\int_{\rnz}Q^\ast_{s,t^m}
(I-P^\ast_{2s_1,t^m})f(x)g(x,t)\frac{\,dx\,dt}{t}.
\end{eqnarray*}
Since $g\in \wtw$, there exist $T_\oz$-atoms $\{a_j\}_j$ and $\{\lz\}_j\subset\cc$
such that $g=\sum_j\lz_ja_j$, where the series converges
in $(\twz)^\ast$, and $\sum_j|\lz_j|\ls\|g\|_{\wtw}$.
For each $k\in\cn$, let $g_k\equiv\sum_{j=1}^k\lz_ja_j$. Then $g_k\in T_2^2(\rnz)
\cap \twl$.
Moreover, by Lemma \ref{l4.8}, we obtain that $\pi_{L,s,2s_1}(g)\in
B_{\oz, L}(\rn)$,
$\pi_{L,s,2s_1}(g_k)\in B_{\oz, L}(\rn)\cap L^2(\rn)$ and
$$ \|\pi_{L,s,2s_1}(g)-\pi_{L,s,2s_1}(g_k)\|_{B_{\oz, L}(\rn)}
\ls\|g-g_k\|_{\wtw}\rightarrow0,$$
as $k\rightarrow\fz$. Thus, by the dominated convergence theorem,
we further obtain that
\begin{eqnarray*}\ell(f)&&=\frac{C_{m,s,2s_1}}{C_{m}}\int_{\rnz}Q^\ast_{s,t^m}
(I-P^\ast_{2s_1,t^m})f(x)g(x,t)\frac{\,dx\,dt}{t}\\
&&\sim\lim_{k\rightarrow\fz}\int_{\rnz}Q^\ast_{s,t^m}
(I-P^\ast_{2s_1,t^m})f(x)g_k(x,t)\frac{\,dx\,dt}{t}\\
&&\sim\lim_{k\rightarrow\fz}\int_{\rn}f(x)\pi_{L,s,2s_1}(g_k)(x)\,dx
\sim\lim_{k\rightarrow\fz} \la f, \pi_{L,s,2s_1}(g_k)\ra\\
&&\sim \sum_j\lz_j\la f,\pi_{L,s,2s_1}(a_j)\ra\sim \la f,\pi_{L,s,2s_1}(g)\ra.
\end{eqnarray*}
Then a density argument via Lemma \ref{l4.10}
completes the proof of Theorem \ref{t4.4}.
\end{proof}

\begin{rem}\label{r4.1}\rm
By Proposition \ref{p4.1} (ii), we know that for any $s\ge s_1\ge \lfr\fss\rf$,
$B_{\oz,L}(\rn)=\hv$. Thus, by Lemma \ref{l4.6},
\eqref{4.9} can be further understood as
\begin{equation*}
\ell(f)\equiv\la f, g\ra=\sum_j \lz_j\int_{\rn} f(x)\az_j(x)\,dx,
\end{equation*}
where $\{\az_j\}_j$ is a sequence of $(\oz,s,s_1)$-molecules
and $\{\lz_j\}_j\in \ell^1$ such that
$g=\sum_j\lz_j\az_j$, where the series converges in $(\bmoz)^\ast$, and $\sum_j|\lz_j|
\sim\|g\|_{B_{\oz,L}(\rn)}$; see also \cite[Theorem B]{cw}.
\end{rem}

\subsection{Several examples of operators satisfying Assumptions (a) and (b)}\label{s4.3}

\hskip\parindent Let us now give several examples of operators
satisfying Assumptions (a) and (b) of this paper, and
hence, all results of this paper are applicable to
these operators.

(i) A typical example of the operator $L$ satisfying
Assumption (a) with $\tz(L)=\fz$ is that the kernels
$\{p_t\}_{t\ge0}$ of $\{e^{-tL}\}_{t\ge 0}$ satisfy a Gaussian upper
bound, namely, there exists a positive constant $C$ such
that for all $x$, $y\in \rn$ and $t > 0$,
$$|p_t(x, y)|\le \frac{C}{t^{n/2}}e^{-\frac{|x-y|^2}{t}}.$$
Obviously, if $L$ is the
Laplacian operator $\Delta$, then the heat kernels satisfy  the Gaussian
upper bound.

Recall that the Neumann Laplacian $\Delta_N$, the
Dirichlet Laplacian $\Delta_D$ and the Dirichlet-Neumann
Laplacian $\Delta_{DN}$ on $\rn$ are extended from
the corresponding operators on half plains;
see \cite[Section 2.2]{ddsy} for details.
It was further shown in \cite{ddsy} that if $L$ is one
of $\Delta_N$, $\Delta_D$ and $\Delta_{DN}$,
then heat kernels of $\{e^{-tL}\}_{t\ge 0}$ satisfy a Gaussian upper
bound, which implies that these operators also
satisfy Assumption (a).

Notice that the operators $\Delta$, $\Delta_N$, $\Delta_D$ and
$\Delta_{DN}$ are self-adjoint operators; hence each of
them has a bounded $H_\fz$-calculus in $L^2(\rn)$ and, therefore,
satisfies Assumption (b); see \cite{m}.

(ii) Let $A\equiv\{a_{ij}(x)\}_{1\le i\le n, 1\le j\le n}$ be an
$n \times n$ matrix with entries $a_{ij}\in L^\fz(\rn, \cc)$
satisfying that for certain $\lz> 0$,
$\mathop\mathrm{Re}(\sum_{i,j=1}^n a_{ij}(x)\xi_i\xi_j)\geq \lz|\xi|^2$
for all $x \in \rn$ and $\xi=(\xi_1, \xi_2, \cdots, \xi_n)\in {\cc}^n$.
Define a divergence form operator
$$Lf\equiv -{\rm div}(A\nabla f),$$
which is  interpreted in the usual weak sense via
a sesquilinear form. Notice that the Gaussian bound of the heat
kernels of $\{e^{-tL}\}_{t>0}$ truly holds
when $A$ has real entries, or when $n = 1, 2$ in the
case of complex entries. Moreover, it is well known that
$L$ admits a bounded $H_\fz$-calculus in $L^2(\rn)$
(see, for example, \cite{at}). Thus, $L$ satisfies
Assumptions (a) and (b).

(iii) Let $V\in L^1_{\loc}(\rn)$ be a nonnegative function. The
Schr\"odinger operator with potential $V$ is defined by $L
\equiv\Delta + V$ on $\rn$, where $n\geq3$ and
$\Delta=-\sum^n_{j=1}(\frac {\partial}{\partial x_j})^2$
is the Laplace operator. From the
Trotter-Kato product formula, it follows that the kernels
$\{p_t\}_{t\ge0}$ of the semigroup $\{e^{-tL}\}_{t\ge0}$ satisfy the
estimate that for all $x$, $y\in \rn$ and $t > 0$,
\begin{equation*}
0\le p_t(x, y)\le \frac{1}{(4\pi t)^{n/2}}e^{-\frac{|x-y|^2}{4t}}.
\end{equation*}
From this, it easy to see that $L$ satisfies Assumptions (a) and (b).
However, unless $V$ satisfies
additional conditions, the heat kernel can be a discontinuous
function of the space variables and the H\"older continuous
estimates may fail to hold; see \cite{da}.

\medskip

{\bf Acknowledgements.} The authors would like
to thank the referee very much for his carefully reading and
several valuable remarks which made this article more readable.

\bigskip

\noindent Renjin Jiang and Dachun Yang (Corresponding author)

\medskip

\noindent School of Mathematical Sciences, Beijing Normal University,
Laboratory of Mathematics and Complex Systems, Ministry of Education,
Beijing 100875, People's Republic of China

\medskip

\noindent{\it E-mail addresses}: \texttt{rj-jiang@mail.bnu.edu.cn}

\hspace{2.5cm}\texttt{dcyang@bnu.edu.cn}

\end{document}